\documentclass[11pt]{amsart}

\usepackage{todonotes}
\usepackage{amsmath}
\usepackage{amssymb}
\usepackage{mathtools}
\usepackage{blindtext}
\usepackage[T1]{fontenc}
\usepackage[utf8]{inputenc}
\usepackage{mathrsfs}
\usepackage{marvosym}
\usepackage{amsthm}
\usepackage{graphicx}
\usepackage{enumitem}
\usepackage[normalem]{ulem}

\usepackage{hyperref}

\usepackage{tikz}
\usetikzlibrary{shapes,positioning,intersections,quotes}

\DeclareMathOperator{\sgn}{sgn}
\DeclareMathOperator{\dist}{dist}
\DeclareMathOperator{\erf}{erf}
\DeclareMathOperator{\supp}{supp}
\DeclareMathOperator{\Lip}{\mathrm{Lip}}
\DeclareMathOperator{\diam}{\mathrm{diam}}
\DeclareMathOperator{\Err}{\mathrm{Err}}
\DeclareMathOperator{\argmin}{\mathrm{argmin}}
\DeclareMathOperator{\conv}{conv}

\newcommand{\eps}{\varepsilon}
\newcommand{\Rd}{\mathbb{R}^d}
\newcommand{\Om}{\Omega}
\newcommand{\SK}{\mathcal{S}}
\newcommand{\Ac}{\mathcal{A}}
\newcommand{\R}{\mathbb{R}}
\newcommand{\subs}{\subseteq}
\newcommand{\sups}{\supseteq}
\newcommand{\setm}{\setminus}
\newcommand{\Sd}{\mathbb{S}^{d-1}}

\newcommand{\tom}{\tilde{\sigma}}
\newcommand{\Oc}{\mathcal{O}}
\newcommand{\ct}{\tilde{\chi}}
\newcommand{\cth}{\tilde{\chi}_h}

\newtheorem{theorem}{Theorem}[section]
\newtheorem{prop}[theorem]{Proposition}
\newtheorem{corollary}[theorem]{Corollary}
\newtheorem{lemma}[theorem]{Lemma}
\newtheorem*{remark}{Remark}

\theoremstyle{definition}
\newtheorem{definition}[theorem]{Definition}
\newtheorem{example}[theorem]{Example}

\title[Thresholding for mean convex sets]{Strong convergence of the thresholding scheme for the mean curvature flow\\ of mean convex sets}

\author{Jakob Fuchs}
\address{Jakob Fuchs, AG Biomathematik, TU Dortmund University, Vogelpothsweg 87, 44227 Dortmund, Germany}
\email{jakob.fuchs@tu-dortmund.de}
	
\author{Tim Laux}
\address{Tim Laux, Hausdorff Center for Mathematics and Institute for Applied Mathematics, University of Bonn, Villa Maria, Endenicher Allee 62, 53115 Bonn, Germany}
\email{tim.laux@hcm.uni-bonn.de}

\begin{document}
	
\begin{abstract}
	In this work, we analyze Merriman, Bence and Osher's thresholding scheme, a time discretization for mean curvature flow. We restrict to the two-phase setting and mean convex initial conditions. In the sense of the minimizing movements interpretation of Esedo\u{g}lu and Otto we show the time-integrated energy of the approximation to converge to the time-integrated energy of the limit. 
	As a corollary, the conditional convergence results of Otto and one of the authors become unconditional in the two-phase mean convex case.
	Our results are general enough to handle the extension of the scheme to anisotropic flows for which a non-negative kernel can be chosen.

	\medskip

	\medskip
	
	\noindent \textbf{Mathematical Subject Classification (MSC 2020)}: 
	   53E10, 
	   35B51 (primary), 
	   65M12, 
	   35A21 
	   (secondary)
\end{abstract}

\maketitle

\section{Introduction}
\label{sec:intro}
	Thresholding is an efficient algorithm for approximating the motion by mean curvature of a surface, which is the $L^2$-gradient flow of the area functional.
	The scheme was introduced by Merriman, Bence, and Osher in 1992 \cite{MBO}. 
	It obtains subsequent time steps by \emph{(i)} diffusion, i.e., convolving with a Gaussian kernel, and \emph{(ii)} thresholding, i.e., taking a super level set.
	We analyze an extension of the scheme to anisotropic mean curvature motion, which arises as weighted gradient descent of normal dependent surface energies. 
	This is achieved by choosing a convolution kernel $K$, different from the original Gaussian $G$. 
	Analogous to the Gaussian we denote the dependence on the time step size of the kernel
	\begin{align*}
	K_h(x) \coloneqq h^{-d/2}K \left( \frac{x}{\sqrt{h}} \right).
	\end{align*}
	The process was shown by Esedo\u{g}lu and Otto \cite{Esedoglu:Otto} to be equivalent to the variational problem (see Lemma~\ref{equiv:rep})
	\begin{align}
	\label{variational}
		E^h_{k+1} \in \underset{E \subs \R^d}{\argmin} \bigg\{& \frac{1}{\sqrt{h}} \int_{E^c} K_h \ast \chi_{E} \,dx \\
	&+ \frac{1}{\sqrt{h}} \int_{\Rd} (\chi_{E_k} -\chi_{E}) K_h \ast (\chi_{E_k} - \chi_E)dx \bigg\}, \notag
	\end{align}
	which can be viewed as an approximate version of minimizing movements, the natural time discretization for gradient flows.
	The first term is to be understood as a surface energy and the second as a mobility cost function. 
	Indeed, the first term was proven by Elsey and Esedo\u{g}lu \cite{Elsey:Esedoglu} to converge to an associated anisotropic perimeter as $h \downarrow 0$ under suitable conditions on the kernel, i.e., 
	\begin{align*}
		\frac{1}{\sqrt{h}} \int_{E^c} K_h \ast \chi_E\,dx \eqqcolon P_{K,h}(E) \rightarrow P_{\sigma_K}(E),
	\end{align*}
	where $P_{\sigma_K}$ is the anisotropic perimeter as defined in the beginning of Section~\ref{sec:initial step} with surface tension $\sigma_K$ as defined in Proposition~\ref{anis:per}. For the Gaussian $G$ we calculate $\sigma_G$ explicitly and denote it by $\sigma$:
\begin{align*}
\sigma_G(x) = \frac{1}{\sqrt{\pi}}|x| = \sigma(x).
\end{align*}
	
	The convergence of the original scheme in the two-phase case to viscosity solutions has been established in the isotropic case by Evans \cite{Evans}, and Barles and Georgelin \cite{Barles:Georgelin}, and in the anisotropic case by Ishii, Pires and Souganidis \cite{Ishii:Pires:Souganidis}. 
	More recently, in the multi-phase case, several conditional convergence results were established, see \cite{Laux:Otto,Laux:Otto:Brakke,Laux:Otto:de:giorgi,Laux:Lelmi}.
	 
	In the following theorem we collect their two-phase expression.
	\begin{theorem}
		\label{Laux}
		Let $E_0 \subs \R^d$ be a set of finite perimeter and $E_h(t)$ the approximation for time step size $h$ by the original thresholding scheme on a finite time horizon $[0,T]$. Then for any sequence there is a subsequence  $h \downarrow 0$, such that $P_{\sigma}(E_{h}(t)) \leq P_{\sigma}(E_0)$ and $E_{h}$ converges to a limit $E$ in $L^1$ and a.e. in space-time. Provided that
		\begin{equation}
		\label{eq:convergence condition}
			\underset{h \downarrow 0}{\lim} \int_0^T P_{G,h}(E_h(t)) dt = \int_0^T P_{\sigma}(E(t)) dt,
		\end{equation}
		the limit $E$ is a weak solution in three senses, namely, a distributional solution, a unit-density Brakke flow, and a De~Giorgi solution.
	\end{theorem}
	
	The main result of this paper, which is stated precisely in Theorem~\ref{energy}, verifies the condition of these convergence results.
	\begin{theorem}
	The convergence of energies \eqref{eq:convergence condition} from Theorem~\ref{Laux} holds if $E_0$ is a finite intersection of bounded, strictly mean convex sets of class $C^{2,\alpha}$ for some $\alpha>0$. Moreover, this convergence of energies result holds for a broad class of not necessarily rotation invariant kernels $K$.
	\end{theorem}

	This is the first result establishing the convergence of energies for the thresholding scheme in the presence of singularities, even in the isotropic case when $K$ is the Gaussian $G$. The class of mean convex surfaces, i.e., $H>0$, arises naturally in mean curvature flow. On the one hand, this property is preserved by the flow \cite[Cor. 3.6]{Huisken}, on the other hand, it allows for interesting singularity formations such as in the famous example of Grayson's dumbbell \cite{Grayson}, where a neck-pinch singularity appears in finite time.
\\

	We compare with the minimizing movements scheme by Almgren, Taylor and Wang \cite{Almgren:Taylor:Wang}, where
	\begin{align*}
	E_{k+1} \in \underset{E \subs \R^d}{\argmin} \left\{ P(E) + \frac{1}{h}\int_{E \triangle E_k} \dist(x,\partial E_k) dx \right\}.
	\end{align*}
	For this scheme a similar conditional convergence result was established by Luckhaus and Sturzenhecker \cite{Luckhaus:Sturzenhecker} and the convergence of energies in the mean convex case was proven by De~Philippis and one of the authors \cite{Laux:Philippis}. 
	The latter result was extended to anisotropic motion by Chambolle and Novaga \cite{Chambolle:Novaga} and to nonlocal perimeters by Cesaroni and Novaga \cite{Cesaroni:Novaga}.
	While, on a conceptual level, the result of this paper plays a similar role as the ones of \cite{Laux:Philippis,Chambolle:Novaga, Cesaroni:Novaga}, the actual mechanism behind these results and ours differ greatly.
	The thresholding scheme is associated with different surface energies and, more importantly, a different movement limiter term, which results in unrelated approximation dynamics.
	In particular, the proofs of \cite{Laux:Philippis, Chambolle:Novaga, Cesaroni:Novaga} heavily rely on the structure of the movement limiter term, in particular the (signed) distance function. 
	In contrast to \cite{Laux:Philippis}, we need to work with quantities that are intrinsic to the scheme, as for example the thresholding energies $\frac1{\sqrt{h}} \int_{E^c} K_h \ast \chi_E$ instead of the perimeter. 
	Another new idea in the present work is the quantification of shrinkage in each time step by a simple geometric comparison principle.
	\\
	
	The rest of this paper is structured as follows.
	In Section~\ref{sec:properties}, we introduce the $K$-outward minimizing condition, which is linked closely to mean convexity. We also provide a useful alternative expression of this condition.
	In Section~\ref{sec:contraction}, we study how sets contract under thresholding. In particular, we exploit the non-negativity of the kernel to show that shrinking sets have monotonically increasing velocity.
	The explicit link between the $K$-outward minimizing condition and contraction under thresholding with $K$ is also established.
	The equivalent expression of the condition is used to study how outward information is maintained by subsequent steps. 
	This allows us to focus on the first time step. 
	Up to this point the assumptions on the kernel are very general and no regularity is needed. 
	In Section~\ref{sec:initial step}, we expand results of Mascarenhas \cite{Mascarenhas} and Evans \cite{Evans} on uniform convergence of motion on smooth surfaces for the Gaussian, and further extend this to anisotropic kernels, following results of Elsey and Esedo\u{g}lu \cite{Elsey:Esedoglu}. After establishing the connection between kernel and approximated motion we extend the initial velocity result to a broad class of kernels, not necessarily rotationally invariant. The proofs of the statements made in this section can be found in Section~\ref{sec:initial step proof}. In Section~\ref{sec:convergence}, we collect information on the approximation and prove the main result of this paper, the convergence of energy.
	Finally, in Section~\ref{sec:kernel construction}, we construct kernels, satisfying all our conditions, that correspond to a large class of motions. This can be made more explicit in two dimensions. We additionally study some fine properties of thresholding.

\section{Properties of the perimeter approximation}
\label{sec:properties}
We first state the precise conditions on the generalized convolution kernel.
\begin{definition}
\label{suit:K}
A function $K$: $\R^{d} \rightarrow \R$ is a \emph{suitable convolution kernel} if it is non-negative, even and integrable, i.e.,
\begin{align*}
K\geq 0, \quad K(x) = K(-x) , \quad
\int_{\R^{d}} K(x) \,dx = 1.
\end{align*} 
\end{definition}
\begin{remark}
Indeed, we can generalize $K$ to be a measure, weakly satisfying the above conditions, and maintain all results where no additional requirements for $K$ are given. Those are Sections~\ref{sec:properties} and~\ref{sec:contraction}, and Lemma~\ref{intersection}.
\end{remark}
Throughout this paper we will work with measurable sets. For a set $D$ we denote the collection of all measurable subsets of $D$ by $\mathcal{M}(D)$ and write $\mathcal{M}=\mathcal{M}(\R^d)$. Suitable convolution kernels being even we note that
\begin{align*}
\int_A K \ast \chi_B \,dx = \int_B K \ast \chi_A \,dx.
\end{align*}
We call this property "reciprocity of impact". To suitable convolution kernels we can associate the $K$-perimeter $P_K$, where for all measurable sets $D \subs \R^d$
\begin{align*}
P_K(D) \coloneqq \int_{D^c} K \ast \chi_D \,dx.
\end{align*}
We define the thresholding operator $T_K$ on $\mathcal{M}$ in the following equivalent ways:
\begin{align}\label{eq:def_thresholding}
T_KE \coloneqq &\{x \in \R^{d} | (K \ast \chi_{E})(x) > 1/2\} \\
= &\{x \in \R^{d} | (K \ast \chi_{E})(x) > (K \ast \chi_{E^c})(x)\}.
\notag
\end{align}
For $E_0 \in \mathcal{M}$, we write $E_1 \coloneqq T_KE_0$. 
We observe that $P_K$ and $T_K$ are invariant under null set perturbations. They are equally defined on equivalence classes. Then the classical thresholding scheme of Merriman, Bence and Osher \cite{MBO} is the special case, where
\begin{align*}
K = G_h = \frac{1}{(4\pi h)^{d/2}}\exp\Big(-\frac{|\cdot |^2}{4h}\Big).
\end{align*}

We verify the variational representation from \eqref{variational}. This is a simple extension of a result by Esedo\u{g}lu and Otto \cite[(5.5)]{Esedoglu:Otto}.
\begin{lemma}
\label{equiv:rep}
Let $K$ be a suitable convolution kernel in the sense of Definition~\ref{suit:K} and $h>0$. Then the set $T_{K_h}E_k$ from \eqref{eq:def_thresholding} is a minimizer of \eqref{variational} (of minimal $\mathcal{L}^d$-measure). 
\end{lemma}
\begin{proof}
Using reciprocity of impact and that $(K_h \ast 1)=1$, we obtain
\begin{align*}
&\underset{E \subs \R^d}{\argmin} \left\{ \frac{1}{\sqrt{h}} \int_{E^c} K_h \ast \chi_E \, dx + \frac{1}{\sqrt{h}} \int_{\R^d} (\chi_{E_k}-\chi_E) K_h \ast (\chi_{E_k}-\chi_E) dx \right\}\\
=\, &\underset{E \subs \R^d}{\argmin} \left\{ \frac{1}{\sqrt{h}} \int_E \left( 1 - 2 K_h \ast \chi_{E_k} \right) dx + \frac{1}{\sqrt{h}} \int_{E_k} K_h \ast \chi_{E_k} dx \right\}\\
\ni\, &\{ x\in \R^d | (K_h \ast \chi_{E_k})(x) > 1/2\} = T_{K_h} E_k.
\end{align*}
Indeed, up to null sets, any minimizer contains this set.
\end{proof} 
While this minimizer is unique (up to Lebesgue null sets) for the Gaussian $G$, this does not hold even for regular classes of general kernels. We discuss this in Subsection~\ref{subsec:fine properties} at the end of the paper.\\

We are first interested in the outward minimizing condition.

\begin{definition}
\label{suff}
Let $K$ be a suitable convolution kernel in the sense of Definition~\ref{suit:K}. A set $E \subseteq \Omega \subseteq \R^{d}$ is called \emph{$K$-outward minimizing in $\Om$} if for all $F \subseteq \Om$
\begin{align*}
P_K(E) \leq P_K(E \cup F).
\end{align*}
We call $\Om$ the \emph{container}. $\Om$ is called \emph{sufficiently large} if $\Om \supseteq T_KE \cup E$.
\end{definition}

For the rest of the section we establish an equivalent formulation to the $K$-outward minimizing condition.

\begin{prop}
\label{omc:eq}
For $K$ a suitable convolution kernel, and $E \subseteq \Om \subseteq \R^{d}$ the following are equivalent:
\begin{equation} \label{omc:1}
\text{For all } F \subseteq \Om \text{: } P_K(E) \leq P_K(E \cup F).
\end{equation}
\begin{equation} \label{omc:2}
\text{For all } G \subseteq \Om \text{: } P_K(E \cap G) \leq P_K(G).
\end{equation}
\end{prop}

To prove this, we extend a result from the proof of Lemma A.3 in \cite{Esedoglu:Otto} to our class of kernels. 
\begin{lemma}
\label{alt:per}
For $K$ a suitable convolution kernel and $E \subs \R^d$ it holds that
\begin{equation}
\label{eq:3}
P_K(E) = \frac{1}{2} \int_{\Rd} K(z) \int_{\Rd} |\chi_E(x) - \chi_E(x-z)| dx dz.
\end{equation}
\end{lemma}
\begin{proof}
By definition of the perimeter approximation
\begin{align*}
P_K(E) &= \int_{E^c} (K \ast \chi_E)(x) dx \\
&= \frac{1}{2} \int_{\Rd} (1-\chi_E)(x) (K \ast \chi_E)(x) + (1-\chi_E)(x) (K \ast \chi_E)(x) dx.
\end{align*}
We use reciprocity of impact on the second term
\begin{align*}
\frac{1}{2} \int_{\Rd} (1-\chi_E)(x) (K \ast \chi_E)(x) + \chi_E(x) (K \ast (1-\chi_E))(x) dx.
\end{align*}
Writing out the convolution, we use Fubini and pull the convolution kernel out of the inner integral
\begin{align*}
\frac{1}{2} \int_{\Rd} K(z) \int_{\Rd} (1-\chi_E)(x)\chi_E(x-z) + \chi_E(x)(1-\chi_E)(x-z) dx\, dz.
\end{align*}
Remembering characteristic functions to be invariant under squaring, we regard the inner integral
\begin{align*}
&(1-\chi_E)(x)\chi_E(x-z) + \chi_E(x)(1-\chi_E)(x-z) \\
= \,&(\chi_E(x))^2 - 2\chi_E(x)\chi_E(x-z) + (\chi_E(x-z))^2 \\
= \,&|\chi_E(x) - \chi_E(x-z)|^2 \\
= \,&|\chi_E(x) - \chi_E(x-z)|.
\end{align*}
This shows \eqref{eq:3}.
\end{proof}

The previous result is applied to prove the central submodularity, which is known for the classical perimeter, to our $K$-perimeter.
\begin{lemma}
\label{per:prop}
For $A, B \subseteq \R^{d}$ and $K$ a suitable convolution kernel
\begin{align*}
P_K(A \cap B) + P_K(A \cup B) \leq P_K(A) + P_K(B).
\end{align*}
\end{lemma}

\begin{proof}
By Lemma~\ref{alt:per}, the claim is equivalent to
\begin{align}
\int_{\Rd} K(z) &\int_{\Rd}  \big(|\chi_{E \cap F}(x) - \chi_{E \cap F}(x-z)| + |\chi_{E \cup F}(x) - \chi_{E \cup F}(x-z)| \big) dx\, dz  \notag \\
\leq &\int_{\Rd} K(z) \int_{\Rd} \big(|\chi_{E}(x) - \chi_{E}(x-z)| + |\chi_{F}(x) - \chi_{F}(x-z)| \big) dx\, dz.
\end{align}
This inequality follows from the elementary inequality for arbitrary real numbers $a,a',b,b'$:
\begin{equation}
\label{eq:6}
|\min(a,b) - \min(a',b')| + |\max(a,b) - \max(a',b')| \leq |a-a'| + |b-b'|.
\end{equation}
Indeed, setting $a = \chi_E(x)$, $a' = \chi_E(x-z)$, $b = \chi_F(x)$ and $b' = \chi_F(x-z)$ and using the non-negativity of $K$ we obtain the claim.\\

We prove the pointwise inequality \eqref{eq:6} as a one dimensional geometric problem: 
Regard $(A,B,C,D)$, a relabeling of $(a,a',b,b')$ according to order, arbitrary in case of equality. Both sides of \eqref{eq:6} are sums of distances in disjoint pairings of $A,B,C,D$. There are only three such pairings:
\begin{align*}
	(AB,CD), (AC,BD), (AD,BC).
\end{align*}
Looking at our number line, we can see that
\begin{align*}
	|AD| + |BC| = |AC| + |BD| = |AB| + |CD| + 2|BC| \geq |AB| + |CD|.
\end{align*}
Now, let us assume the statement \eqref{eq:6} to be wrong:
\begin{align*}
	&|\min(a,b) - \min(a',b')| + |\max(a,b) - \max(a',b')| > |a - a'| + |b - b'|.
\end{align*}
We then can associate $|a-a'| + |b-b'|$ only with $|AB| + |CD|$, as all other sums of distances are maximal. Thus
\begin{align*}
	\{a,a'\} =\{ A,B\} \text{ and } \{ b,b'\} =\{ C,D\}
\end{align*}
or
\begin{align*}
 	\{a,a'\} =\{ C,D\} \text{ and } \{ b,b'\} =\{ A,B\}.
\end{align*}
W.l.o.g.\ we assume the first case. But then, the distance of minima is between $a$ and $a'$, that of maxima between $b$ and $b'$. This implies
\begin{align*}
	|\min(a,b) - \min(a',b')| + |\max(a,b) - \max(a',b')| &= |AB| + |CD|\\
	&= |a-a'| + |b-b'|,
\end{align*}
contradicting the assumption.
\end{proof}

\begin{proof}[Proof of Proposition~\ref{omc:eq}]
To prove that \eqref{omc:2} implies \eqref{omc:1}, we set $G = E \cup F$ in \eqref{omc:2}. For the inverse implication, let $G \subs \Om$. By \eqref{omc:1} we immediately obtain
\begin{align*}
P_K(E \cap G) + P_K(E) \leq P_K(E \cap G) + P_K(E \cup G)
\end{align*}
and by Lemma~\ref{per:prop}
\begin{align*}
P_K(E \cap G) + P_K(E \cup G) \leq P_K(G) + P_K(E).
\end{align*}
Thus
\begin{align*}
P_K(E \cap G) + P_K(E) \leq P_K(G) + P_K(E)
\end{align*}
and by subtracting $P_K(E)$ we conclude the proof.
\end{proof}
Noting the trivial connection between $F$ and $G$, we can extend this result slightly.
\begin{corollary}\label{cor:per}
For $E \subseteq \Om \subseteq \R^{d}$, $K$ a suitable convolution kernel and $\mathcal{A}$ an operator on the subsets of $\Om$. Then the following statements are equivalent:
\begin{align}
\label{eq:11} &\text{For all } F \subseteq \Om \text{: } P_K(E) + \mathcal{A}(F \setminus E) \leq P_K(E \cup F) \\
\label{eq:12} &\text{For all } G \subseteq \Om \text{: } P_K(E \cap G) + \mathcal{A}(G \setminus E) \leq P_K(G)
\end{align}
\end{corollary}
\begin{proof}
This proof is fully analogous to the last. To prove that \eqref{eq:12} implies \eqref{eq:11} we set $G = E \cup F$ in \eqref{eq:12}. For the inverse implication, let $G \subs \Om$. By \eqref{eq:11} we immediately obtain
\begin{align*}
P_K(E \cap G) + P_K(E) + \mathcal{A}(G \setminus E) \leq P_K(E \cap G) + P_K(E \cup G)
\end{align*}
and by Lemma~\ref{per:prop}:
\begin{align*}
P_K(E \cap G) + P_K(E \cup G) \leq P_K(G) + P_K(E).
\end{align*}
Thus
\begin{align*}
P_K(E \cap G) + P_K(E) + \mathcal{A}(G \setminus E) \leq P_K(G) + P_K(E)
\end{align*}
and by subtracting $P_K(E)$ we conclude the proof.
\end{proof}

\section{Contraction under thresholding}
\label{sec:contraction}

In this section, we investigate the contraction behavior of thresholding. 
More precisely, we show in Subsection~\ref{subsec:monotone speed} that the speed of contracting sets is non-decreasing. 
In Subsection~\ref{subsec:outward min contract}, we show the equivalence of contraction under thresholding and an almost outward minimality condition.
To denote the sequence of sets obtained by repeated application of thresholding we write $E_1 \coloneqq T_KE_0$ and $E_j \coloneqq T_K^j E_0 = T_K(\dots(T_KE_0)\dots)$.

\subsection{Monotone speed}
\label{subsec:monotone speed}

The following theorem states that sets which contract under thresholding do so with monotonically increasing speed. 

\begin{theorem}
	\label{mon:vel}
	Let $K$ be a suitable convolution kernel in the sense of Definition~\ref{suit:K}, $E_0 \subs \R^{d}$ be measurable and write $E_j = T_K^jE_0$. If $E_1 \subs E_0$, then $E_2 = T_KE_1 \subs E_1$ and
	\begin{align*}
	\dist(E_2, \partial E_1) \geq \dist(E_1,\partial E_0).
	\end{align*}
\end{theorem}

\begin{proof}
	We use a set comparison argument. For a visual representation see Figure~\ref{fig:comparison}. If $E_1 \subs E_0$ and $\dist(E_1,\partial E_0) = r$, then for all $s \in B_r(0)$ we have $(E_1 + s) \subs E_0$. Thus for all $x \in \R^{d}$
	\begin{align*}
	\chi_{E_1}(x-s) \leq \chi_{E_0}(x).
	\end{align*}
	Since $K$ is non-negative this implies
	\begin{align*}
	(K \ast \chi_{E_1})(x-s) \leq (K \ast \chi_{E_0})(x).
	\end{align*}
	Hence, by definition of thresholding, if $x \in E_2$, then 
	\begin{equation}
	\label{eq:shifted comparison}
	\frac{1}{2} < (K \ast \chi_{E_1})(x) \leq (K \ast \chi_{E_0})(x+s).
	\end{equation}
	But then $(x+s) \in E_1$. Since $s \in B_r(0)$ was arbitrary, this means that the distance between $E_2$ and $\partial E_1$ is larger or equal to $r = \dist(E_1, \partial E_0)$.
\end{proof}

\begin{remark}
Note that the simple proof of Theorem~\ref{mon:vel} did not use the symmetry $K(-x)=K(x)$. Furthermore, replacing the value $1/2$ in \eqref{eq:shifted comparison} by $1/2 - c$, the result also extends to thresholding with constant drift $c \in \R$, in which $T_KE$ is replaced by $\{x \in \R^d | (K \ast \chi_E)(x) > 1/2 - c\}$.
\end{remark} 

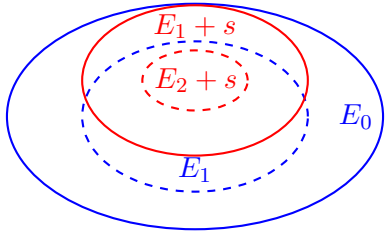
\begin{figure}
	\centering
	\begin{tikzpicture}
	
	\draw [blue, thick] (0,0) ellipse (2.5 and 1.5);
	
	\node [blue, left=0pt of {(2.5,0)}] {$E_0$};
	
	\draw [blue, dashed, thick] (0,0) ellipse (1.5 and 1);
	
	\node [blue, above=0pt of {(0,-1)}] {$E_1$};
	
	\draw [red, thick] (0,.48) ellipse (1.5 and 1);
	
	\node [red, below=0pt of {(0,1.5)}] {$E_1+s$};
	
	\draw [red, dashed, thick] (0,.48) ellipse (0.7 and 0.4);
	
	\node [red, left=0pt of {(0.7,.48)}] {$E_2+s$};
	
	\end{tikzpicture}
	\caption{Proof by comparison with shifted successor}
	\label{fig:comparison}
\end{figure}

\subsection{Outward minimizing sets contract}
\label{subsec:outward min contract}

This section connects contraction of a set under a thresholding step with $K$ to the $K$-outward minimizing condition. The basis of this is the following lemma.  The proof being a simple calculation, we can carry along an arbitrary functional $\Ac$.
\begin{lemma}\label{main:con}
For $E, F \subseteq \R^{d}$ disjoint with $P_K(E)<\infty$, $K$ a suitable convolution kernel in the sense of Definition~\ref{suit:K} and $\Ac = \Ac(E,F)$ an arbitrary functional the following statements are equivalent:
\begin{align}
\label{eq:7} P_K(E \cup F) + \int_F K \ast \chi_F \,dx &\geq P_K(E) + \Ac(E,F),\\
\label{eq:8} \int_{F} K \ast \chi_{E^c} \,dx &\geq \int_F K \ast \chi_E \,dx + \Ac(E,F).
\end{align}
\end{lemma}
\begin{remark}
We call the second term on the left-hand side of \eqref{eq:7} the self-interaction of $F$ and denote it by
\begin{align*}
\SK_K(F) \coloneqq \int_F K \ast \chi_F \,dx.
\end{align*}
This term will appear several times in this section. It has first variation $0$, but for any fixed $F$
\begin{align*}
\mathcal{S}_{\frac{1}{\sqrt{h}}K_h}(F) \rightarrow \infty \text{ for } h \rightarrow 0.
\end{align*}
That we can weaken the outward minimizing condition by this term points to thresholding depending on the surface structure in a local sense, despite the convolution step's dependency on the global structure of the set.
\end{remark}
\begin{proof}[Proof of Lemma \ref{main:con}]
Let
\begin{align*}
P_K(E \cup F) + \SK_K(F) \geq P_K(E) + \Ac(E,F).
\end{align*}
By definition of the $K$-perimeter and the self-interaction this is equivalent to
\begin{align*}
\int_{(E \cup F)^c} K \ast \chi_{E \cup F} \,dx + \int_F K \ast \chi_F \,dx \geq \int_{E^c} K \ast \chi_E \,dx + \Ac(E,F).
\end{align*}
With $E$ and $F$ being disjoint, we can split the terms up
\begin{align*}
&\int_{(E \cup F)^c} K \ast \chi_{E} \,dx + \int_{(E \cup F)^c} K \ast \chi_F \,dx + \int_F K \ast \chi_F \,dx\\
&\geq \int_{(E \cup F)^c} K \ast \chi_E \,dx + \int_F K \ast \chi_E \,dx + \Ac(E,F)
\end{align*}
and subtract the first summand (using $P_K(E)<\infty$)
\begin{align*}
\int_{E^c} K \ast \chi_F \,dx \geq \int_F K \ast \chi_E \,dx + \Ac(E,F).
\end{align*}
Using reciprocity of impact, we conclude
\begin{align*}
\int_F K \ast \chi_{E^c} \,dx \geq \int_F K \ast \chi_E \,dx + \Ac(E,F). & \qedhere
\end{align*}
\end{proof}
From Lemma~\ref{main:con} we immediately gain the explicit connection between the $K$-outward minimizing condition and contraction under thresholding. 

\begin{theorem}
\label{contraction}
Let $E_0 \subseteq \Omega \subseteq \R^{d}$ with $P_K(E_0)<\infty$, $K$ be a suitable convolution kernel in the sense of Definition~\ref{suit:K} and $\Omega$ be sufficiently large in the sense of Definition~\ref{suff}. Then $|E_1 \setminus E_0|=0$ if and only if for all $F \subseteq \Om \setminus E_0$
\begin{align*}
P_K(E_0 \cup F) + \SK_K(F) \geq P_K(E_0).
\end{align*}
\end{theorem}
\begin{proof}
The second statement is by the previous Lemma~\ref{main:con} equivalent to 
\begin{align*}
\int_F K \ast \chi_{E_0^c} \,dx \geq \int_F K \ast \chi_{E_0} \,dx \text{, for all } F \subs \Om \setm E_0.
\end{align*}
This is equivalent to
\begin{align*}
(K \ast \chi_{E_0^c})(x) \geq (K \ast \chi_{E_0})(x) \text{, for almost all } x \in \Om \setm E_0.
\end{align*}
Since $E_1 = \{x \in \Om | (K \ast \chi_{E_0^c})(x) < (K \ast \chi_{E_0})(x)\}$, this is equivalent to $|E_1 \setminus E_0|=0$.
\end{proof}

\begin{definition}
If $E_0$, $K$ and $\Om$ satisfy the conditions and equivalent statements of Theorem~\ref{contraction}, we call $E_0$ \emph{locally $K$-outward minimizing in $\Om$}.
\end{definition}

We now show that $E_1$ inherits local $K$-outward minimality from $E_0$. 
\begin{theorem}
\label{per:contr}
Let $E_0$ be locally $K$-outward minimizing in $\Om$, then, for all $G \subseteq \Om \setminus E_1$
\begin{align*}
P_K(E_1 \cup G) + \SK_K(G) \geq P_K(E_1) + 2\int_G K \ast \chi_{E_0 \setminus E_1} dx.
\end{align*}
In particular, $E_1$ is also locally $K$-outward minimizing in $\Om$.
\end{theorem}
\begin{proof}
By Theorem~\ref{contraction} $E_1 \subseteq E_0$. Let $G \subseteq E_1^c$. By definition of thresholding
\begin{align*}
\int_G K \ast \chi_{E_0^c} \,dx \geq \int_G K \ast \chi_{E_0} \,dx.
\end{align*}
Since $\chi_{E_0^c} = \chi_{E_1^c} - \chi_{E_0 \setm E_1}$ and $\chi_{E_0} = \chi_{E_1} + \chi_{E_0 \setm E_1}$, this is equivalent to 
\begin{align*}
\int_G K \ast \chi_{E_1^c} \,dx \geq \int_G K \ast \chi_{E_1} \,dx + 2\int_G K \ast \chi_{E_0 \setminus E_1} \,dx.
\end{align*}
By Lemma~\ref{main:con} this is equivalent to 
\begin{align*}
P_K(E_1 \cup G) + \SK_K(G) \geq P_K(E_1) + 2\int_G K \ast \chi_{E_0 \setminus E_1} \,dx.
\end{align*}
Then, $E_1$ being locally $K$-outward minimizing in $\Om$ follows immediately from
\begin{align*}
2\int_G K \ast \chi_{E_0 \setminus E_1} \,dx \geq 0. & \qedhere
\end{align*}
\end{proof}

As an immediate consequence, if $E_0$ contracts under thresholding, its $K$-perimeter decreases as well. This is true in general, if $K$ has non-negative Fourier transform \cite{Esedoglu:Otto}, but in our case, this assumption is not needed.
\begin{corollary}
\label{des:ener}
Let $E_0$ be locally $K$-outward minimizing in $\Om$. Then
\begin{align*}
P_K(E_0) - \SK_K(E_0 \setminus E_1) \geq P_K(E_1).
\end{align*}
\end{corollary}
\begin{proof}
Set $G$ to $E_0 \setminus E_1$ in Theorem~\ref{per:contr}.
\end{proof}

If we have additional information on comparison with supersets of $E_0$, we can maintain them somewhat for $E_1$. To prove this, we use the equivalent formulation of the outward minimizing condition established in Proposition~\ref{omc:eq}. This also proves that $E_1$ inherits $K$-outward minimality from $E_0$.

\begin{corollary}
\label{ext:omc}
Let $E_0 \subseteq \Omega \subseteq \R^{d}$ with $P_K(E_0)<\infty$, $K$ be a suitable convolution kernel in the sense of Definition~\ref{suit:K} and $\Omega$ be sufficiently large in the sense of Definition~\ref{suff}. Let $\Ac: \mathcal{M}(\Om \setm E_0) \rightarrow \mathbb{R}$, with $\Ac \geq -\SK_K$. If for all $F \subseteq \Om$:
\begin{align*}
P_K(E_0 \cup F) \geq P_K(E_0) + \Ac(F \setminus E_0),
\end{align*}
then for all $F \subseteq \Om$
\begin{align*}
&P_K(E_1 \cup F) \\
&~~\geq P_K(E_1) + \Ac(F \setminus E_0) + \int_{F \cap (E_0\setminus E_1)} K \ast (2\chi_{E_0 \setminus E_1}-\chi_{F \cap (E_0 \setminus E_1)}) \,dx.
\end{align*}
In particular, if $E_0$ is $K$-outward minimizing in $\Om$, then so is $E_1$.
\end{corollary}

\begin{proof}
As $\Ac \geq -\SK_K$, by Theorem~\ref{contraction} $E_1 \subseteq E_0$. The alternative definition of the outward minimizing conditions allows now to combine information on $\Om \setminus E_0$ and $E_0 \setminus E_1$. Firstly,
\begin{align*}
P_K(E_0 \cup F) \geq P_K(E_0) + \Ac(F \setminus E_0)
\end{align*}
for all $F \subseteq \Om$ is by Corollary~\ref{cor:per} equivalent to 
\begin{align*}
P_K(F) \geq P_K(E_0 \cap F) + \Ac(F \setminus E_0)
\end{align*}
for all $F \subseteq \Om$. Secondly, by Theorem~\ref{per:contr}, for all $G \subseteq \Om \setminus E_1$
\begin{align*}
P_K(E_1 \cup G) + \SK_K(G) \geq P_K(E_1) + 2\int_G K \ast \chi_{E_0 \setminus E_1} \,dx.
\end{align*}
We write equivalently, that for all $G \subseteq \Om$
\begin{align*}
P_K(E_1 \cup G) \geq P_K(E_1) + \int_{G \setminus E_1} K \ast (2\chi_{E_0 \setminus E_1} - \chi_{G \setminus E_1}) dx.
\end{align*}
But this is, again by Corollary~\ref{cor:per}, equivalent to
\begin{align*}
P_K(G) \geq P_K(E_1 \cap G) + \int_{G \setminus E_1} K \ast (2\chi_{E_0 \setminus E_1} - \chi_{G \setminus E_1}) dx
\end{align*}
for all $G \subseteq \Om$. We set $G = F \cap E_0$ and combine. For all $F \subseteq \Om$
\begin{align*}
P_K(F) \geq& \; P_K(F \cap E_0) + \Ac(F \setminus E_0) \\
\geq& \;P_K(F \cap E_0 \cap E_1) + \Ac(F \setminus E_0)\\
&+ \int_{(F \cap E_0) \setminus E_1} K \ast (2\chi_{E_0 \setminus E_1} - \chi_{(F \cap E_0) \setminus E_1}) dx.
\end{align*}
With $E_1 \subseteq E_0$ we can write, that for all $F \subseteq \Om$
\begin{align*}
&P_K(F)
\\&~~~ \geq P_K(F \cap E_1) + \Ac(F \setminus E_0) + \int_{F \cap (E_0 \setminus E_1)} K \ast (2\chi_{E_0 \setminus E_1} - \chi_{F \cap (E_0 \setminus E_1)}) dx.
\end{align*}
The last two summands can be written as an operator on $F \setm E_1$. This allows us to apply Corollary~\ref{cor:per} again, which concludes the proof. This also proves, that if $E_0$ is $K$-outward minimizing ($\Ac = 0$), so is $E_1$, as
\begin{align*}
\int_{F \cap (E_0 \setminus E_1)} K \ast (2\chi_{E_0 \setminus E_1} - \chi_{F \cap (E_0 \setminus E_1)}) dx \geq 0. & \qedhere
\end{align*}
\end{proof}

\section{Initial step and consistency}
\label{sec:initial step}

	In Section~\ref{sec:contraction}, we have shown that for general non-negative kernels $K$, contracting sets keep contracting at a non-decreasing rate.
	This means that we essentially only need to understand the first time step in order to make statements about the whole evolution.\\
	
	This section is devoted to studying this first time step. 
	While the arguments in the previous sections were completely general regarding the convolution kernel $K$, we now have to be more specific.
	In Subsection \ref{subsec:initial gaussian}, we focus on the standard Gaussian kernel, in which case the estimates can be made explicit using the error function. 
	In Subsection \ref{subsec:initial anisotropies}, we treat four large classes of anisotropic kernels with mild decay and regularity properties.
	We stress that in contrast to Section~\ref{sec:contraction}, the results in this section do not rely on the pointwise positivity of the kernel.
	All proofs are given in Section~\ref{sec:initial step proof}.
	
\subsection{The Gaussian case}
\label{subsec:initial gaussian}

In this subsection, we analyze the initial time step in the standard case, where $K$ is the heat kernel
\begin{align*}
K(x) = G_h(x) \coloneqq (4 \pi h)^{-d/2} \exp\left(-\frac{|x|^2}{4h}\right).
\end{align*}
This is a suitable kernel, as $G_h$ is positive, rotation invariant and $\|G_h\|_1 = 1$. Then
\begin{align*}
E_1^h \coloneqq T_{G_h}E_0 = \{x \in \R^{d} | (G_h \ast \chi_{E_0})(x) > 1/2\}.
\end{align*}

We want to show $\dist(E_1^h, \partial E_0) \geq h(1-\eps)H_0$. To obtain this, we prove uniform convergence of the signed distance of the first step. In the following statement, we parametrize the new interface $\partial E_1$ as a graph over the initial interface $\partial E_0$, as can be seen in Figure~\ref{fig:par surface}. We define $z$ as the distance of the boundary in direction of the outer normal. 

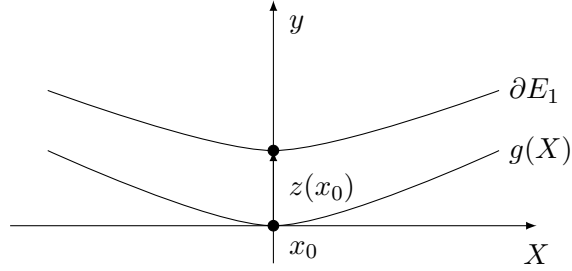
\begin{figure}
\centering
\begin{tikzpicture}

\draw [-latex] (-3.5,0) -- (3.5,0);

\node [below=3pt of {(3.5,0)}] {$X$};

\draw [-latex] (0,-.5) -- (0,3);

\node [below right=3pt of {(0,3)}] {$y$};

\draw plot [smooth] coordinates{(-3,1) (0,0)  (3,1)};

\draw plot [smooth] coordinates{(-3,1.8) (0,1)  (3,1.8)};

\node [right=2pt of {(0,0.5)}] {$z(x_0)$};

\node [below right=3pt of {(0,0)}] {$x_0$};

\draw [-latex] (0,0) -- (0,1);

\filldraw (0,1) circle (2pt);

\filldraw (0,0) circle (2pt);

\node [ right = 0pt of {(3,1)}] {$g(X)$};

\node [ right = 0pt of {(3,1.8)}] {$\partial E_1$};
\end{tikzpicture}
\caption{Parametrization of the surface}
\label{fig:par surface}
\end{figure}

\begin{theorem}
\label{unif:con}
Let $E_0 \subs \R^{d}$ be a bounded set with $C^{2,\alpha}$-boundary, for some $\alpha \in (0,1]$. For $x \in \partial E_0$, let $\nu(x)$ be the outer unit normal of $E_0$ in $x$ and $H(x;E_0)$ be the mean curvature of $\partial E_0$ in $x$. We define
\begin{equation*}
z(x) \coloneqq \begin{cases}
\sup\{l \in \R_-| x + l\nu \in E_1^h\} &\text{ if } x \notin E_1^h,\\
~\inf\{l \in \R_+| x + l\nu \notin E_1^h\} &\text{ if } x \in E_1^h.
\end{cases}
\end{equation*}
Then for $h$ sufficiently small, uniformly in $x$,
\begin{align*}
z(x) = -hH(x;E_0) + \Oc(h^{1+\alpha / 2}).
\end{align*}
Moreover, if $E_0$ has $C^4$-boundary
\begin{align*}
z(x) = -hH(x;E_0) + \Oc(h^{2}).
\end{align*}
\end{theorem}

\subsection{General anisotropies}
\label{subsec:initial anisotropies}

In this subsection, we analyze the initial time step for four classes of general convolution kernels. 
Before extending the statement on initial/smooth steps to a large class of convolution kernels, we study the connection between convolution kernel and approximated movement.\\

Mean curvature flow is the steepest descent for the perimeter of a domain. This perimeter can be expressed as
\begin{equation}
\label{eq:standardperimeter}
P(E) = \int_{\partial^* E}|\nu(x)| d\mathcal{H}^{d-1}(x).
\end{equation}
The measure of surface is independent of direction, or isotropic. If we are to measure the surface weighted dependent on its direction, or anisotropically, we get
\begin{align*}
P_{\sigma}(E) = \int_{\partial^* E} \sigma(\nu(x)) d\mathcal{H}^{d-1}(x),
\end{align*}
for a surface tension $\sigma$: $\Sd \rightarrow \mathbb{R}$ continuous and even. From this we get a\\different mean curvature, the resulting mean curvature flow optimizing descent of this surface measure. We define the adjusted $K_h$-perimeter, which is the time step dependent surface energy measure as discussed in the introduction
\begin{align*}
P_{K,h}(E) \coloneqq \frac{1}{\sqrt{h}}P_{K_h}(E) = h^{-(d+1)/2} \int_{E^c} K\left( \frac{\cdot}{\sqrt{h}} \right) \ast \chi_E \,dx. 
\end{align*}
The connection between surface tension and kernel is given by Elsey and Esedo\u{g}lu in \cite{Elsey:Esedoglu} for compact sets with smooth boundary. We extend the proof analogously to sets of finite perimeter.

\begin{prop}
\label{anis:per}
Let $E \subs \R^{d}$ be a set of finite perimeter and $K$ an even function with $\int_{\Rd} K(x) dx =1$ and
\begin{equation}
\label{K:bound}
K(x) \leq \frac{C}{1+|x|^p}
\end{equation}
for some positive number $C<\infty$ and $p>d+1$. Then, for $\nu$ the outer unit normal
\begin{align*}
\lim_{h \downarrow 0}P_{K,h}(E) = \int_{\partial^* E} \sigma_K(\nu(x)) d\mathcal{H}^{d-1}(x),
\end{align*}
where the surface tension is defined as
\begin{align*}
\sigma_K(\nu) \coloneqq \frac{1}{2} \int_{\R^{d}} |\nu \cdot x| K(x) dx.
\end{align*}
\end{prop}

This gives us the connection between convolution kernel and surface tension. To find a kernel for a specific surface tension we need to solve an inverse problem. We can simplify the inverse problem for a function on the entire space to an inverse problem for a function on the sphere by writing
\begin{align*}
\sigma_K(\nu) &\coloneqq \frac{1}{2} \int_{\R^{d}} |\nu \cdot x| K(x) dx\\
&= \frac{1}{2} \int_{\Sd} |\nu \cdot \theta| \int_0^{\infty} r^d K(r\theta) dr d\mathcal{H}^{d-1}(\theta)
\end{align*}
and we can define the tension generating directional distribution
\begin{align*}
A(\theta) \coloneqq \int_0^{\infty} r^d K(r \theta) dr.
\end{align*}
Thus, we can write the surface tension as
\begin{align*}
\sigma_A(\nu) = \frac{1}{2} \int_{\Sd} |\nu \cdot \theta| A(\theta) d\mathcal{H}^{d-1}(\theta).
\end{align*}
In particular, all kernels $K$ generating the same function $A$ generate the same surface tension. \\

We now derive the anisotropic curvature function from the surface tension. We begin by canonically extending a given surface tension function $\sigma$ defined on the sphere $\Sd$ to the whole space $\R^d$ by
\begin{align*}
\sigma(p) \coloneqq |p|\sigma\left( \frac{p}{|p|}\right) \text{, } p \in \R^d.
\end{align*}
Let $\tau_i$ be the orthonormal basis of eigenvectors of the second fundamental form. This is a basis of the orthogonal complement of the unit vector $\nu$ with $\nabla \nu \tau_i = \kappa_i \tau_i$. If the surface tension function is sufficiently smooth, the curvature function is defined as
\begin{align*}
H_{\sigma}(x;E_0) &= \left( \nabla \cdot (D_p \nu) \right)(x)\\
&= \sum_{i=1}^{d-1} (\tau_i \cdot \nabla) D_p\sigma(\nu) \tau_i \\
&= \sum_{i=1}^{d-1} \tau_i \cdot D_p^2\sigma(\nu) \nabla \nu \tau_i
\end{align*}
and by definition of the principal curvatures
\begin{align*}
\nabla \nu \tau_i = \kappa_i \tau_i.
\end{align*}
Thus, we can write
\begin{align*}
H_{\sigma}(x;E_0) = \sum_{i=1}^{d-1} \kappa_i (\tau_i \cdot D_p^2\sigma(\nu) \tau_i).
\end{align*}
Now we want to choose a kernel $K$ generating a given surface tension $\sigma$. Being able to rotate $K$ arbitrarily, we set $\nu(x)=e_d$ and $\tau_i = e_i$. Then, the proof of Proposition 6 in \cite{Elsey:Esedoglu} gives us the following result.

\begin{prop}
\label{div:sigma}
Let $K$ satisfy the conditions of Proposition~\ref{anis:per} and 
\begin{equation}
\label{eq:convergence of quadratic}
\underset{\eps \rightarrow 0}{\lim} \underset{O \in \mathrm{O}(d)}{\sup} \int_{\R^{d-1}} |K \circ O(X,\eps f(X)) - K \circ O(X,0)| (1 + |X|^2) dX = 0,
\end{equation}
with $f(X) = (X \cdot MX)\, +\, a$ for a fixed symmetric $(d-1)$-matrix $M$ and number $a$. Then
\begin{align*}
D^2\sigma_{K}(x) = \frac{1}{|x|} \int_{\{y \cdot x = 0\}} y \otimes y K(y) d \mathcal{H}^{d-1}(y).
\end{align*}
\end{prop}

If $K$ satisfies the conditions, we can write the anisotropic mean curvature in terms of the kernel
\begin{equation}
\label{anisotropic:curvature:expression}
H_{\sigma_{K}}(x;E_0) = \sum_{i=1}^{d-1} \kappa_i \int_{\R^{d-1}} X_i^2 K(X,0) dX.
\end{equation}
Lastly, \cite{Elsey:Esedoglu} gives the mobility function $\mu_K$: $\Sd \rightarrow \R$ such that
\begin{align*}
\frac{1}{\mu_K(n)} = 2\int_{\{x \cdot n = 0\}} K(x) d\mathcal{H}^{d-1}(x).
\end{align*}
We can represent this in polar coordinates to establish the mobility generating directional distribution
\begin{align*}
B(\theta) \coloneqq 2\int_0^{\infty} r^{d-2}K(r\theta) dr
\end{align*}
on which mobility depends such that
\begin{align*}
\frac{1}{\mu_K(n)} = \int_{ \{ \theta \in \Sd | \theta \cdot n = 0\}} B(\theta) d\mathcal{H}^{d-2}(\theta).
\end{align*}
Then the movement approximated by thresholding only depends on the two directional distributions $A$ and $B$, allowing for infinitely many kernels to approximate the same movement.\\

Having established the connection between the quantities in mean curvature flow and the corresponding ones in the context of thresholding, we can now state the main result of this section: The velocity in the initial thresholding step converges to the initial velocity of the corresponding mean curvature flow uniformly. 
The result holds firstly for two classes of non-negative kernels, either we assume that $K$ is H\"older continuous with compact support, or we assume Lipschitz continuity and some decay at infinity. In the latter case we also need to impose the positivity of the Radon transform.

\begin{theorem}
\label{anis:main}
Let $E_0 \subs \R^{d}$ be a bounded set with $C^{2,\alpha}$-boundary, for some $\alpha \in (0,1]$, and K a suitable convolution kernel in the sense of Definition~\ref{suit:K} satisfying one of the following conditions:
\begin{enumerate}[label=(\Roman*)]
\item $K \in C_c^{0,\alpha}(\R^{d})$,
\item $K \in C^{0,1}(\R^{d})$ and there are positive numbers $C$, $\eps$ and $\delta$, such that for all $x$, $\overline{x} \in \R^{d}$ and $n \in \Sd$
\end{enumerate}
\begin{align*}
|K(x)| &\leq C|x|^{-d-2-\eps},\\
|K(x) - K(\overline{x})| &\leq C\min(|x|, |\overline{x}|)^{-d-3-\delta} |x-\overline{x}|,\\
\frac{1}{\mu_K(n)} &= 2\int_{\{x \cdot n =0\}} K(x) d\mathcal{H}^{d-1}(x) > 0.
\end{align*}
For $E_1^{K,h} = \{x \in \R^{d} | (K_h \ast \chi_{E_0})(x) > 1 /2\}$ the result of a thresholding step with $K_h$ and $\nu(x)$ the outer normal vector of a boundary point $x$, we again define the discrete normal movement function of the surface
\begin{equation*}
z(x) \coloneqq \begin{cases}
\sup\{l \in \R_-| x + l\nu \in E_1^{K,h}\} &\text{ if } x \notin E_1^{K,h}\\
~\inf\{l \in \R_+| x + l\nu \notin E_1^{K,h}\} &\text{ if } x \in E_1^{K,h}.
\end{cases}
\end{equation*}
Then, 
\begin{align}
	\label{eq:initial movement PDE}
	\frac{1}{\mu_K(\nu(x))} z(x) = -h H_{\sigma_{K}}(x;E_0) + \mathcal{O}(h^{1+\alpha/2}).
\end{align}
\end{theorem}

The number $\alpha \in (0,1]$ in the theorem concerns in both instances the degree of H\"older-continuity of functions with compact support. 
For simplicity, we do not distinguish these two numbers as $\alpha$ could simply be taken as the minimum of the two H\"older exponents. 
The result and its proof are related to Proposition 13 in \cite{Elsey:Esedoglu}, where the limit is proven pointwise on smooth surfaces in $\R^3$.\\

The result on initial motion represents the consistency of thresholding as a numerical process. It thus is interesting in its own right. As this result does not rely on the comparison principle, we do not need to assume positivity of the kernel. Conditions on the Fourier transform of the kernel create interest in functions that are partially negative. We can indeed weaken the condition of non-negativity for the initial motion result. This expands our consistency result for the initial time step not only to kernels with minor oscillations around $0$, but most surprisingly even to kernels inverting classical mean curvature flow, which will be discussed in Section~\ref{sec:kernel construction}.

\begin{theorem}
\label{negative:Ker}
Let $E_0 \subs \R^d$ be a bounded set with $C^{2,\alpha}$-boundary for some $\alpha \in (0,1]$. Let $K$ be an even function with $\int_{\R^d}K(x) dx =1$ and for all $\theta \in \mathbb{S}^{d-1}$ and $r>0$
\begin{align*}
\int_0^r s^{d-1}K(s\theta)ds \geq 0.
\end{align*}
Let $K$ further satisfy one of the following conditions:
\begin{enumerate}[label=(\Roman*)]
\item $K \in C_c^{0,\alpha}(\R^d)$, and there is a $\tilde{\theta} \in \mathbb{S}^{d-1}$ such that for all $r>0$
\begin{align*}
\int_0^r s^{d-1}K(s\tilde{\theta})ds >0.
\end{align*}
\item $K \in C^{0,1}(\R^d)$ with $K(0)>0$ and there are positive numbers $C<\infty$ and $\eps, \delta >0$ such that for all $x,\overline{x} \in \R^d$
\begin{align*}
|K(x)| &\leq C|x|^{-d-2-\eps},\\
|K(x)-K(\overline{x})| &\leq C\min(|x|,|\overline{x}|)^{-d-3-\delta}|x-\overline{x}|.
\end{align*}
\end{enumerate}
For $E_1^{K,h} = \{x \in \R^d | (K_h \ast \chi_{E_0})(x) > 1/2\}$ the result of a thresholding step with $K_h$ and $\nu(x)$ the outer unit normal vector of a boundary point $x$, we define the discrete normal movement function of the surface
\begin{equation*}
z(x) \coloneqq \begin{cases}
\sup\{l \in \R_-| x + l\nu \in E_1^{K,h}\}  &\text{if } x \notin E_1^{K,h}\\
~\inf\{l \in \R_+| x + l\nu \notin E_1^{K,h}\} &\text{if } x \in E_1^{K,h}.
\end{cases}
\end{equation*}
Then
\begin{align*}
\frac{1}{\mu_K(\nu(x))} z(x) = -h H_{\sigma_K}(x;E_0) + \mathcal{O}(h^{1+\alpha/2}).
\end{align*}
\end{theorem}

We should note, that while there is no kernel satisfying the given conditions that generates negative mobility, strangely enough, there are such kernels that generate negative surface tensions. In Subsection~\ref{subsec:backwards-in-time}, we discuss kernels that generate the ``surface tension''
\begin{align*}
\sigma_K = -|\cdot|
\end{align*}
and hence lead to backwards-in-time mean curvature flow. While our consistency of the present section remains true in that case, the scheme will be unstable.\\

Several times in the next chapter we will need to bound the adjusted $K_h$-perimeter. We conclude the chapter by bounding it by its limit anisotropic perimeter. This extends a result by Esedo\u{g}lu and Otto \cite{Esedoglu:Otto} for a class of rotation invariant kernels.

\begin{prop}\label{prop:monotonicity of energies}
Let $K$ be a kernel satisfying the conditions of Proposition~\ref{anis:per} with
\begin{align*}
\int_0^{\infty} r^d K(r\theta)dr \geq 0
\end{align*}
for all $\theta \in \Sd$ and $E$ be a set of finite perimeter. Then, for all $h>0$
\begin{align*}
P_{K,h}(E) \leq P_{\sigma_K}(E).
\end{align*}
\end{prop}

\section{Convergence}
\label{sec:convergence}
The uniform convergence of the initial movement derived in Section~\ref{sec:initial step} implies that strictly mean convex domains contract. By Section~\ref{sec:contraction}, we know that such contraction is maintained by further steps of the process. This will allow us to prove the strong convergence of the scheme as stated in the central result of this section, Theorem~\ref{energy}. We first collect the precise information gained, and conclude by proving the desired convergence results.\\

First we define the arrival time function (see Figure~\ref{fig:arrival time}) and the piecewise constant interpolation of our discrete scheme, which will approximate mean curvature flow.
\begin{definition}
Let $h>0$ and $(E_k^h)_{k \in \mathbb{N}}$ be a result of the thresholding scheme. Then $u_h$: $\overline{E_0}\rightarrow \R_{\geq 0}$, with
\begin{align*}
u_h \coloneqq h\sum_{k \in \mathbb{N}_{0}} \chi_{E_k^h} 
\end{align*}
is the arrival time function, and
\begin{align*}
E_h(t) \coloneqq E_k^h \text{ , for } t \in h[k,k+1)
\end{align*}
is the piecewise constant in time interpolation of our sets $E_k^h$.
\end{definition}

We notice that thresholding is, with perfect generality, submultiplicative and superadditive.
\begin{lemma}
\label{intersection}
For arbitrary measurable sets $A_1, A_2 \subs \R^{d}$ and $K$ a suitable convolution kernel in the sense of Definition~\ref{suit:K}, the following statements are true:
\begin{align*}
(1) \;T_K(A_1 \cap A_2) &\subs T_KA_1 \cap T_KA_2,\\
(2) \;T_K(A_1 \cup A_2) &\sups T_KA_1 \cup T_KA_2.
\end{align*}
\end{lemma}
\begin{proof}
(1) Remembering the definition of thresholding $x \in T_K(A_1 \cap A_2)$ is equivalent to
\begin{align*}
\int_{A_1 \cap A_2} K(x-y) dy > 1/2.
\end{align*}
But $A_1$ and $A_2$ are supersets of $A_1 \cap A_2$ and with $K$ non-negative we obtain for $i=1,2$
\begin{align*}
\int_{A_i} K(x-y) dy > 1/2.
\end{align*}
This immediately implies $x \in T_KA_1 \cap T_KA_2$.\\
(2) Let now $x \in T_KA_1 \cup T_KA_2$ and w.l.o.g.\  $x \in T_KA_1$. Then $A_1 \subs A_1 \cup A_2$ implies $x \in T_K(A_1 \cup A_2)$.
\end{proof}

\begin{figure}
\centering
\begin{tikzpicture}
\filldraw[gray!10] (-0.5, 1.9365) arc (104.48: 255.52: 2) -- (-0.5, -1.9365) arc (-75.52: 75.52: 2);

\filldraw[gray!20] (-0.5, 1.308) arc (110.92: 249.08: 1.4) -- (-0.5, -1.308) arc (-69.08: 69.08: 1.4);

\draw [blue!40, thick] (-1,0) circle (2);

\draw [red!40, thick] (0,0) circle (2);

\node [above right=1pt of {(-3,0)}] {$A_1$};

\node [above left=0pt of {(2,0)}] {$A_2$};

\draw [blue!60, thick] (-1,0) circle (1.4);

\draw [red!60, thick] (0,0) circle (1.4);

\node [below = 4pt of {(-.5,2)}] {$E_0$};

\filldraw[gray!30] (-.5,0) ellipse (.8 and 1);

\draw[gray!70] (-.5,0) ellipse (.8 and 1);

\node [below = 0pt of {(-0.5,0.36)}] {$E_1$};
\end{tikzpicture}
\caption{Thresholding on intersection of mean convex sets}
\label{fig:intersection}
\end{figure}
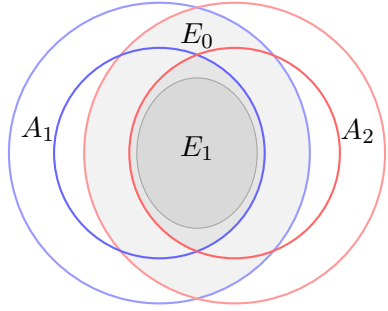

We remember the definition of the adjusted $K_h$-perimeter
\begin{align*}
P_{K,h}(E) = \frac{1}{\sqrt{h}}P_{K_h}(E)
\end{align*}
and define analogously the self-interaction
\begin{align*}
\mathcal{S}_{K,h}(E) \coloneqq \frac{1}{\sqrt{h}} \mathcal{S}_{K_h}(E) = h^{-(d+1)/2} \int_E K \left( \frac{\cdot}{\sqrt{h}} \right) \ast \chi_E \,dx.
\end{align*}
We consider the property of strict mean convexity and collect its implications.

\begin{theorem}
\label{collect}
Let $N \in \mathbb{N}$, $\alpha > 0$ and $(A_l)_{l=1}^N$ a finite family of bounded sets of class $C^{2,\alpha}$. Let $K$ be a convolution kernel satisfying the conditions of Theorem~\ref{anis:main}, inducing a bounded mobility function $\mu_K$ and a surface tension function $\sigma$ and for all $l$ and $x \in \partial A_l$, $H_{\sigma}(x;A_l) > 0$. Let further 
\begin{align*}
E_0 \coloneqq \bigcap_{l=1}^N A_l,
\end{align*} 
and $E_k^h \coloneqq T_{K_h}^kE_0$ the result of $k$ thresholding steps on $E_0$. Then, there are positive numbers $h_0$, $w$ and $T_0$, such that for all $h \in (0,h_0)$ and all $T \geq T_0$
\begin{align*}
E_{k+1}^h &\subs E_k^h,\\
\dist(E_{k+1}^h, \partial E_k^h) &\geq wh, \\
E_{\left \lfloor{T/h}\right \rfloor}^h &= \varnothing.
\end{align*} 
Subsequently, any countable subset of the family of arrival time functions $\{u_h\}_{h\in (0,h_0)}$, with indices converging to zero, is relatively compact in the $C^0$-topology and any accumulation point is in $C^{0,1}(\overline{E_0},\R_{\geq 0})$. Furthermore
\begin{align*}
\int_0^{\infty} P_{K,h}(E_h(t)) dt = h \sum_{k=0}^{\infty} P_{K,h}(E_k^h) \leq T_0P_{K,h}(E_0) \leq 2CT_0\sum_{l=1}^NP(A_l),
\end{align*}
with $P$ the standard perimeter functional, see \eqref{eq:standardperimeter}.
\end{theorem}

\begin{proof}
We first want to show that the intersection $E_0$ contracts by some positive distance. By Theorem~\ref{anis:main}, for any  $l$ and all points $x$ on the boundary of $A_l$, the normal distance to $\partial T_{K_h}A_l$ divided by $h$ converges uniformly to $-\mu_K(x;A_l)H_{\sigma}(x;A_l)$. Since $A_l$ has compact boundary, there is a $h_0^l>0$, such that for all $h \in (0,h_0^l)$, $T_{K_h}A_l \subs A_l$ and for some positive constant $w_l$ and $\eps>0$
\begin{align*}
\dist(T_{K_h}A_l, A_l^c) \geq h\, (1-\eps) \underset{x}{\min}(H_{\sigma}(x;A_l)) \geq w_lh.
\end{align*}
Thus for $h_0 \coloneqq \min_l(h_0^l)$, $h \in (0,h_0)$ and $w \coloneqq \min_l(w_l)$,
\begin{align*}
\underset{l}{\min}(\dist(T_{K_h}A_l, A_l^c)) \geq wh.
\end{align*}
We can also write, that for all $l$, if $x \in T_{K_h}A_l$ and $y \in B_{wh}(0)$, then $x+y \in A_l$. But this implies that,
\begin{align*}
\text{if } x \in \bigcap_{l=1}^N T_{K_h}A_l \text{ and } y \in B_{wh}(0) \text{, then } x + y \in \bigcap_{l=1}^N A_l.
\end{align*}
This again is equivalent to
\begin{align*}
\dist \left(\bigcap_{l=1}^N T_{K_h}A_l, \left(\bigcap_{l=1}^N A_l\right)^c \right) \geq wh,
\end{align*}
and by Lemma~\ref{intersection}, we know 
\begin{align*}
T_{K_h}\left( \bigcap_{l=1}^N A_l \right) \subs \bigcap_{l=1}^N T_{K_h}A_l.
\end{align*}
Monotonicity of the distance function under "$\subs$" allows us to conclude the proof of initial movement
\begin{align*}
\dist(E_1^h, \partial E_0) &= \dist \left( T_{K_h}\bigcap_{l=1}^N A_l, \left(\bigcap_{l=1}^N A_l\right)^c \right)\\
&\geq \dist \left(\bigcap_{l=1}^N T_{K_h}A_l, \left(\bigcap_{l=1}^N A_l\right)^c \right) \geq wh.
\end{align*}
This means that these intersections of sets move further inwards than the sets from which they were constructed, which is visualized in Figure~\ref{fig:intersection}. Then, by Theorem~\ref{mon:vel}, for all $k \in \mathbb{N}$, we have $E_{k+1}^h \subs E_k^h$ and
\begin{align*}
\dist(E_{k+1}^h, \partial E_k^h) \geq wh.
\end{align*}
We now want to prove the existence of an upper time-bound, for $E_h(t) \neq \varnothing$. We set
\begin{align*}
T_0 \coloneqq \frac{\diam(E_0)}{w}
\end{align*}
and monitor the distance between the original set and the result of $\left \lfloor{T_0/h}\right \rfloor$ thresholding steps:
\begin{align*}
\dist(E_{\left \lfloor{T_0/h}\right \rfloor}^h, \partial E_0) &\geq \sum_{i=1}^{\left \lfloor{T_0/h}\right \rfloor} \dist(\partial E_i^h, \partial E_{i-1}^h)\\
&\geq \left \lfloor{\diam(E_0)/(wh)}\right \rfloor wh \geq (2/3)\diam(E_0).
\end{align*}
Since $E_{\left \lfloor{T_0/h}\right \rfloor}^h \subs E_0$, this implies $E_{\left \lfloor{T/h}\right \rfloor}^h = \varnothing$ for all $T>T_0$. \\

Next, we want to find a bound of the time-integral of the $K$-perimeter of the approximated process. Since $P_{K,h}(\varnothing)=0$, the last result implies
\begin{align*}
\int_0^{\infty} P_{K,h}(E_h(t)) dt = h\sum_{k=0}^{\infty} P_{K,h}(E_k^h) = h \sum_{k=0}^{\lfloor T_0/h \rfloor} P_{K,h}(E_k^h).
\end{align*}
It remains to show that $P_{K,h}(E_k^h)$ is decreasing in $k$. For all $k \in \mathbb{N}$, $E_{k+1}^h \subs E_k^h$ implies by Theorem~\ref{contraction} that $E_k^h$ is locally $K_h$-outward minimizing. Then, by Corollary~\ref{des:ener},
\begin{align*}
P_{K,h}(E_{k+1}^h) \leq P_{K,h}(E_k^h) - \mathcal{S}_{K,h}(E_k^h \setm E_{k+1}^h)
\end{align*}
and, since $\mathcal{S}_{K,h}$ is non-negative, for all $k$ 
\begin{align*}
P_{K,h}(E_k^h) \leq P_{K,h}(E_0).
\end{align*}
Then we can estimate
\begin{align*}
\int_0^{\infty} P_{K,h}(E_h(t)) dt = h \sum_{k=0}^{\lfloor T_0/h \rfloor} P_{K,h}(E_k^h) \leq T_0 P_{K,h}(E_0).
\end{align*}






\begin{figure}
	\pgfmathsetmacro {\nb}{6}
	\begin{tikzpicture}[scale=1.5]
		\draw[->] (-2.5,0) -- (2.5,0) node[below]{$x$};
		\draw (-2.2,{1/\nb})--(-2.3,{1/\nb}) node[left]{$h$};
		\draw[->] (-2.25,-.25) -- (-2.25,2.5) node[left]{$u_h(x)$};
		\draw [domain=-2:2,variable=\x,samples=500] plot ({\x},{floor(\nb*2.3*(.9-1/4*.9*\x*\x))/\nb});
	\end{tikzpicture}
	\caption{Discrete arrival time function}
	\label{fig:arrival time}
\end{figure}
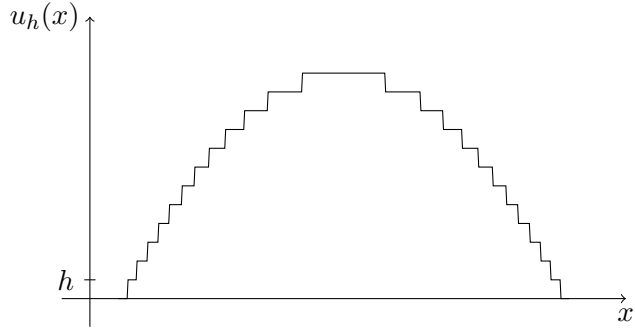

It remains to show that the $K$-perimeter of the initial domain is uniformly bounded. By \eqref{eq:weak perimeter bound} and since $K \geq 0$:
\begin{align*}
P_{K,h}(E_0) &= \frac{1}{\sqrt{h}} \int_{\Rd} K_h(y) \int_{\Rd} \chi_{E_0}(x-y) \chi_{E_0^c}(x) dx dy\\
&= \int_{\Rd} K(y) \frac{1}{\sqrt{h}} \int_{\Rd} \chi_{E_0}(x - \sqrt{h}y) \chi_{E_0^c}(x) dx dy\\
&\leq \int_{\Rd} K(y) |y| P(E_0) dy \leq C(K) P(E_0) \leq C(K) \sum_{l=1}^N P(A_l) < \infty.
\end{align*}
Next we want to show that the family of arrival time functions is relatively compact. To prove this, we use a generalization of the Arzelà-Ascoli theorem for non-continuous functions \cite{Dro:Eym}. We know that $(\overline{E_0},|\cdot - \cdot|)$ is a compact metric space and $(\R_{\geq 0},|\cdot - \cdot|)$ is a complete metric space. We need to show that there exists a number C, such that for all $h \in (0,h_0)$ and $x, \overline{x} \in \R^{d}$
\begin{align*}
|u_h(x) - u_h(\overline{x})| \leq C|x - \overline{x}| + h.
\end{align*}
That is $u_h$, having the structure described in Figure~\ref{fig:arrival time}, satisfies a large scale Lipschitz condition. All values of $u_h$ are multiples of $h$. Let w.l.o.g.\  $u_h(x) = nh$ and $u_h(\overline{x}) = (m+n)h$ for $m,n \in \mathbb{N}$. Then 
\begin{align*}
|u_h(x) - u_h(\overline{x})| = mh
\end{align*}
and $x \in E_n^h \setm E_{n+1}^h$, $\overline{x} \in E_{n+m}^h \setm E_{n+m+1}^h$. Since $E_{k+1}^h \subs E_k^h$ for all $k$, any line from $x$ to $\overline{x}$ has to cross all intermediate boundaries and 
\begin{align*}
|x - \overline{x}| \geq \dist(\partial E_{n+1}^h, \partial E_{n+m}^h) \geq \sum_{l=n+1}^{n+m-1} \dist(\partial E_l^h, \partial E_{l+1}^h) \geq (m-1)wh.
\end{align*}
Setting $C = 1/w$ we obtain
\begin{align*}
\frac{1}{w}|x - \overline{x}| + h \geq (m - 1)h + h = mh = |u_h(x) - u_h(\overline{x})|.
\end{align*}
Thus any countable subset with indices converging to zero is relatively compact in the $C^0$-topology and any accumulation point is a continuous function. Let $u$ be the limit of such a sequence $u_{h_k}$. Due to the uniform convergence for any two points $x$ and $\overline{x}$ we can find a number $k$, such that
\begin{align*}
|u_{h_k}(x) - u(x)| + |u_{h_k}(\overline{x}) - u(\overline{x})| < \frac{|x-\overline{x}|}{w} - h_k.
\end{align*} 
Then we obtain
\begin{align*}
|u(x) - u(\overline{x})| &\leq |u_{h_k}(x) - u(x)| + |u_{h_k}(x) - u_{h_k}(\overline{x})| + |u_{h_k}(\overline{x}) - u(\overline{x})| \\
&\leq \frac{2}{w}|x-\overline{x}|.
\end{align*}
Thus $u$ is Lipschitz continuous.
\end{proof}

We now can state the main convergence result of this paper.

\begin{theorem}
\label{energy}
Let $K$ satisfy the conditions of Theorem~\ref{anis:main}, inducing a bounded mobility function $\mu_K$ and let $E_0$ satisfy the conditions of Theorem~\ref{collect}. Then there exists a function $E$: $\R_{\geq 0} \rightarrow \mathcal{M}(E_0)$ such that $E_{h}(t) \to E(t)$ in $L^1$ for a.e. $t>0$ and such that the energies converge:
\begin{align*}
\underset{h \downarrow 0}{\lim} \int_0^{\infty} P_{K,h}(E_{h}(t)) dt 
&= \int_0^{\infty} P_{\sigma_K}(E(t)) dt .
\end{align*}
\end{theorem}

\begin{proof}
By Theorem~\ref{collect}, for the family of arrival time functions $\{u_h\}_{h\in(0,h_0)}$ every sequence $h \rightarrow 0$ has a subsequence converging to some function $u \in C^{0,1}(E_0)$. 
 We choose one such sequence and set $E(t) \coloneqq \{x \in E_0 | u(x) > t\}$. First we prove that for almost all $t > 0$, $E_{h_k}(t)$ converges to $E(t)$ in $L^1$. Since $u_{h_k}$ converges to $u$ uniformly, $\left( \{u > t\} \setminus \{u_{h_k} > t\} \right)$ and $\left( \{u < t\} \setminus \{u_{h_k} < t\} \right)$ both converge to the empty set in $L^1$. By a result of Alberti, Bianchini and Crippa \cite{Alberti:Bianchini:Crippa}, since $u$ is Lipschitz continuous and compactly supported, for almost all $t>0$ we have $\mathcal{H}^{d-1}(\{u = t\}) < \infty$ and thus $\mathcal{L}^d(\{u = t\}) = 0$, which implies the $L^1$ convergence.\\

We prove that we can apply \cite[Theorem 3.3]{Ishii:Pires:Souganidis} so that the limit $E(t)$, generated by our class of kernels, is the unique viscosity solution starting from $E_0$. Since this identifies the limit, everything we state in the following also holds for the whole sequence $h\downarrow0$.
The kernels are suitable in the sense of Definition~\ref{suit:K}. Since the mobility function $\mu_K$ is bounded, its inverse is positively bounded from below and by the given rate of decay we obtain
\begin{align*}
0 < \int_{\{(x \cdot n)=0\}} (1+|x|^2)K(x)dx < \infty.
\end{align*}
From \eqref{eq:mobility continuity} in the proof of Lemma~\ref{anis:z:bound} we obtain continuity of $1/\mu_K$ for the second class of kernels. Indeed the result still holds for $x_ix_jK(x)$. On a compact support the uniform continuity of the kernel is sufficient. Lastly, the convergence of quadratic surfaces is a simple consequence of \eqref{eq:convergence of quadratic}, which is shown to hold in the proof of Theorem~\ref{anis:main}.\\

We first want to show the crucial upper bound
\begin{equation}\label{eq:crucial upper bound}
\underset{h \downarrow 0}{\limsup} \int_0^{\infty} P_{K,h}(E_{h}(t))dt \leq \int_0^{\infty} P_{\sigma_K}(E(t)) dt.
\end{equation}
By Theorem~\ref{collect}, $E_0$ contracts under thresholding with $K_h$ for all $h \in (0,h_0)$. Thus $E_0$ is a sufficiently large container of $E_0$ and all subsets of $E_0$ in the sense of Definition~\ref{suff}. Trivially, $E_0$ is $K_h$-outward minimizing in $E_0$. By applying Corollary~\ref{ext:omc} inductively we obtain that $E_k^h$ is outward minimizing in $E_0$ for all $k \in \mathbb{N}$ and $h \in (0,h_0)$, i.e. for all $F \subs E_0$
\begin{align*}
P_{K,h}(E_k^h) \leq P_{K,h}(E_k^h \cup F).
\end{align*}
Let $\eps > 0$ and $h$ be small enough such that $\|u - u_h\|_{\infty} < \eps$, cf. Theorem~\ref{collect}. Since $\{x \in \Rd | u(x)>0\} \subs \{x \in \Rd | u_h(x)>0\} = E_0$ we have
\begin{align*}
v \coloneqq u + \eps \chi_{E_0} \geq u_h
\end{align*}
and obtain for $t \in h[k,k+1)$ that
\begin{align*}
E_k^h = \left\{ x \in \R^d \Big| u_h(x) > t \right\} \subs \left\{ x \in \R^d \Big| v(x) > t \right\} \subs E_0.
\end{align*}
The $K_h$-outward minimality of $E_k^h$ in $E_0$ then implies
\begin{align*}
P_{K,h}(\{u_h > t\}) \leq P_{K,h}(\{ v > t \}).
\end{align*}
But the set $\{v > t\}$ does not depend on $h$ and, with $K_h$ non-negative, we can apply Proposition~\ref{prop:monotonicity of energies} to obtain
\begin{align*}
P_{K,h}(\{ v > t \}) \leq P_{\sigma_K}(\{v > t\}).
\end{align*}
With the right-hand side independent of $h$ we can now estimate the limit superior
\begin{align*}
\underset{h \downarrow 0}{\limsup} \int_0^{\infty} P_{K,h}(\{u_h >t\}) dt \leq \int_0^{\infty} P_{\sigma_K}(\{v >t\}) dt.
\end{align*}
Since $\{u > t\} \subs E_0$, by the construction of $v$ we note that for $t<\eps$ we have $\{v > t\}=E_0$ and for $t > \eps$ we have $\{v > t\}=\{u>t-\eps\}$. Thus
\begin{align*}
\int_0^{\infty} P_{\sigma_K}(\{v>t\})dt = \eps P_{\sigma_K}(E_0) + \int_0^{\infty} P_{\sigma_K}(\{u>t\})dt.
\end{align*}
Therefore, translating this again into a statement for our sets $E_h(t)$ and $E(t)$, we obtain
\begin{align*}
\underset{h \downarrow 0}{\limsup} \int_0^{\infty} P_{K,h}(E_h(t)) dt \leq \eps P_{\sigma_K}(E_0) + \int_0^{\infty} P_{\sigma_K}(E(t)) dt.
\end{align*}
Since $P_{\sigma_K}(E_0)$ is finite and $\eps$ is arbitrarily small this concludes the argument for \eqref{eq:crucial upper bound}.
\\

Next we prove the lower bound
\begin{align*}
\underset{h \downarrow 0}{\liminf} \int_0^{\infty} P_{h}(E_h(t)) dt \geq \int_0^{\infty} P_{\sigma_K}(E(t)) dt;
\end{align*}
a general result that does not depend on $E_h$ coming from thresholding. 
From a result of Alberti and Bellettini \cite{Alb:Bell:asym:beh} we obtain below that for sets $A_h$ converging to a set $A$ in $L^1$
\begin{align*}
\underset{h \downarrow 0}{\liminf} P_{K,h}(A_h) \geq P_{\sigma_K}(A).
\end{align*}
Since for almost all $t > 0$, $E_{h}(t)$ converges to $E(t)$, this implies
\begin{align*}
\underset{h \downarrow 0}{\liminf} P_{K,h}(E_{h}(t)) \geq P_{\sigma_K}(E(t))
\end{align*}
for almost all $t > 0$. Then by Fatou's Lemma
\begin{align*}
\underset{h \downarrow 0}{\liminf} \int_0^{\infty} P_{K,h}(E_{h}(t)) dt &\geq \int_0^{\infty} \underset{h \downarrow 0}{\liminf} P_{K,h}(E_{h}(t)) dt 
\\&\geq \int_0^{\infty} P_{\sigma_K}(E(t)) dt.
\end{align*}

In the rest of this proof, we make sure that the result of \cite{Alb:Bell:asym:beh} can indeed be applied in our situation and we post-process it by characterizing their cell formula for the surface tension in our case.
Theorem 1.4 (ii) of \cite{Alb:Bell:asym:beh} states that for every $\ct \in \mathrm{BV}(\R^d, \{-1,+1\})$ and every sequence $(\cth)$ such that $\cth$ converges to $\ct$ in $L^1(\Om)$, we have
\begin{align*}
\underset{h \downarrow 0}{\liminf} \, F_h(\cth) \geq F(\ct)
\end{align*}
for
\begin{align*}
F_h(\ct) \coloneqq \frac{1}{4\sqrt{h}} \int_{\Rd} \int_{\Rd} K_h(x-y)(\ct(x) - \ct(y))^2 dx dy + \frac{1}{\sqrt{h}} \int_{\Rd} W(\ct(x)) dx,
\end{align*}
where $W$ is any continuous double-well potential vanishing only at $\pm 1$ and growing at least linearly at infinity and $K$ satisfying the conditions of Proposition~\ref{anis:per} with $K \geq 0$, and
\begin{align*}
F(\ct) \coloneqq \int_{\partial^* \{\ct(x)=1\}} \tom(\nu(x)) d\mathcal{H}^{d-1}(x).
\end{align*}
Defining $\tom$ later, we can apply this to our problem by restricting the functions $\ct_h$ to take values in $\{-1,+1\}$ as well. Then, due to the double-well potential vanishing in $\pm 1$,
\begin{align*}
P_{K,h}(E) = \frac{1}{2} F_h(2\chi_E-1),
\end{align*}
and we obtain
\begin{align*}
\underset{h \downarrow 0}{\liminf}\, P_{K,h}(E_{h}(t)) \geq \frac{1}{2} \int_{\partial^* E(t)} \tom(\nu(x)) d\mathcal{H}^{d-1}(x).
\end{align*}
It remains to prove that $\tom = 2\sigma_K$. First we recall the definition of $\tom$ from \cite{Alb:Bell:asym:beh}. For $e$ a unit vector in $\R^d$ choose $(e_1, \ldots , e_{d-1})$ such that $(e_1, \ldots, e_{d-1},e)$ is an orthonormal basis. Let $M$ be the orthogonal complement of $e$. We define $\mathcal{C}_e$ as the class of all $(d-1)$-dimensional cubes on  $M$ centered on $0$. For every cube $C \in \mathcal{C}_e$, we call the strip on the cube $T_C \coloneqq \{X + te | X\in C, t \in \R\}$. We denote by $X(C)$ the class of $C$-periodic functions $\ct$: $\R^d \rightarrow [-1,1]$ that satisfy
\begin{align*}
\underset{t \rightarrow \infty}{\lim}\, \ct(X, \pm t) = \pm 1.
\end{align*}
Then, for
\begin{align*}
\mathcal{F}(\ct,A) \coloneqq \frac{1}{4} \int_A \int_{\R^d} K(y) (\ct(x+y) - \ct(x))^2 dy dx + \int_A W(\ct(x)) dx 
\end{align*}
we set
\begin{align*}
\tom(e) \coloneqq \inf\{|C|^{-1} \mathcal{F}(\ct,T_C): C \in \mathcal{C}_e, \ct \in X(C)\}.
\end{align*}
To find this infimum we collect results from \cite{Alb:Bell:opt:pro} of Alberti and Bellettini. By Theorem 3.3 of \cite{Alb:Bell:opt:pro}, for any cube $C$, $\mathcal{F}(\ct,T_C)$ is minimized by a function $\gamma^e$ that varies only in the direction of $e$. Moreover the minimum is independent of the size of the cube $C$ and 
\begin{align*}
\min \left\{ |C|^{-1}\mathcal{F}(\ct,T_C) | C \in \mathcal{C}_e , \ct \in X(C) \right\} = \min \left\{ \mathcal{F}^e(\gamma^e) | \gamma^e \in \mathrm{BV}(\R;[-1,1]) \right\}
\end{align*}
for
\begin{align*}
\mathcal{F}^e(\gamma^e) \coloneqq \frac{1}{4} \int_{\R} \int_{\R} \int_M K(x + se) (\gamma^e(s+t) - \gamma^e(t))^2 dx ds dt + \int_{\R} W(\gamma^e(t))dt.
\end{align*}
Having reduced the problem to a one-dimensional one, we can apply the results of Section 2.19 of \cite{Alb:Bell:opt:pro}. As we have reduced the problem to a function class independent of the double-well potential $W$, we can choose it freely. If we choose 
\begin{align*}
W(t) \geq c(1-t^2) \text{ for } t\in [-1,1],
\end{align*}
where
\begin{align*}
c = \int_{\{(x\cdot e) \geq 0\}} K(x) dx = \frac{1}{2},
\end{align*}
then the minimizer is the one-dimensional sign function. This allows us to explicitly express the infimum as
\begin{align*}
\tom(e) = \inf \left\{ |C|^{-1}\mathcal{F}(\ct,T_C) | C \in \mathcal{C}_e, \ct \in X(C) \right\} = \mathcal{F}^e(\sgn).
\end{align*}
We first use the symmetry of $K$ and then apply Fubini's Theorem twice to obtain
\begin{align*}
\mathcal{F}^e(\sgn) &= \frac{1}{4} \int_{\R} \int_{\R} \int_M K(x+(s-t)e) d\mathcal{H}^{d-1}(x) (\sgn(s) - \sgn(t))^2 ds\, dt\\
&= 2 \int_{\R_+} \int_{\R_+} \int_{\{(x \cdot e) = 0\}} K(x + (s+t)e) d\mathcal{H}^{d-1}(x)\, ds\, dt\\
&= 2 \int_{\R_+} \int_{\{(x \cdot e) \geq t\}} K(x) dx\, dt \\
&= 2 \int_{\{(x \cdot e) \geq 0\}} \int_{\R_+} \chi_{\{x \cdot e \geq t\}} dt\, K(x) dx.
\end{align*}
Calculating the inner integral we can again use that $K$ is an even function to conclude
\begin{align*}
\mathcal{F}^e(\sgn) &= 2 \int_{\{(x \cdot e) \geq 0\}} (x \cdot e) K(x) dx\\
&= \int_{\R^d} |x \cdot e| K(x) dx\\
&= 2 \, \sigma_K(e). \qedhere
\end{align*}
\end{proof}

\section{Proofs of statements in Section~\ref{sec:initial step}}
\label{sec:initial step proof}

In Section~\ref{sec:initial step} we established the uniform convergence of motion in the initial step. Having ommited all proofs there, they are given here in order.

\begin{proof}[Proof of Theorem~\ref{unif:con}]
We fix $x_0 \in \partial E_0$. Following an idea of Mascarenhas \cite{Mascarenhas}, as $E_0$ has $C^{2,\alpha}$-boundary, we find a neighborhood U of $x_0$, a coordinate system with $x_0=0$ and $\nu(x_0)=e_d$ and a function $g$: $\R^{d-1} \rightarrow \R$, such that $g(0) = \partial_{x_i} g(0) = 0$ and in this coordinate system
\begin{align*}
E_0 \cap U = \{(X,y) \in \R^{d-1}\times \R | y < g(X)\} \cap U.
\end{align*}
By the spectral theorem, the symmetric matrix $D^2g(0)$ is diagonalized by an orthonormal basis of eigenvectors. We fix these as $(e_i)_{i=0}^{d-1}$ in the coordinate system and define a function 
\begin{align*}
F(X,y,h) \coloneqq (G_h \ast \chi_{E_0})(X,y)
\end{align*}
and solve $F(0,z,h)=1/2$ for $z \in \R$. The heat kernel decreasing exponentially, $G(x) < C|x|^{-d-3}$ for some $C<\infty$. Hence $G_h(x) < Ch^{-d/2}|x/\sqrt{h}|^{-d-3}$,
\begin{align*}
\int_{U^c} G_h(x) dx < C h^{3/2} \int_{U^c} |x|^{-d-3} dx < C'h^{3/2}
\end{align*}
and the difference outside of $U$ is of order $h^{3/2}$. Then we express $F(0,z,h)$ in terms of the function $g$:
\begin{align*}
\frac{1}{2} = (4 \pi h)^{-d/2}\int_{\R^{d-1}}\int_{-\infty}^{g(X)}\exp\left(-\frac{|X|^2+(y+z)^2}{4h}\right)dydX + \Oc(h^{3/2}).
\end{align*}
By properties of the exponential function and substitution, we obtain that
\begin{align*}
\frac{1}{2} = (4 \pi h)^{-d/2}\int_{\R^{d-1}}\exp\left(-\frac{|X|^2}{4h}\right)\int_{-\infty}^{g(X)+z}\exp\left(-\frac{y^2}{4h}\right)dydX + \Oc(h^{3/2}).
\end{align*}
This is the sum of the integral over the half-space and the integral from $0$ to $g(X)+z$. The first one being $1/2$, we subtract it and use the (lower) error function.
\begin{align*}
0 = (4 \pi h)^{-d/2}\int_{\R^{d-1}}\exp\left(-\frac{|X|^2}{4h}\right)\sqrt{\pi h}\erf\left(\frac{g(X)+z}{\sqrt{4h}}\right)dX + \Oc(h^{3/2}).
\end{align*}
The error function is an entire function with power series representation
\begin{align*}
\erf(x) = \frac{2}{\sqrt{\pi}}\sum_{k=0}^{\infty}\frac{(-1)^k}{(2k+1)k!}x^{2k+1}.
\end{align*}
We will see that both $g(X)/\sqrt{h}$ and $z/\sqrt{h}$ vanish in the limit $h\to 0$. Since we are interested in the leading order, we first regard only the summand for $k=0$
\begin{align*}
0 = (4 \pi h)^{-d/2}\int_{\R^{d-1}}\exp\left(-\frac{|X|^2}{4h}\right)(g(X)+z)dX,
\end{align*}
and as $z$ does not depend on $X$, this simply means
\begin{align*}
-(4 \pi h)^{-1/2} z = (4 \pi h)^{-d/2}\int_{\R^{d-1}}\exp\left(-\frac{|X|^2}{4h}\right)g(X)dX.
\end{align*}
We multiply with $-\sqrt{4\pi h}$ and will remember to do the same with the rest:
\begin{equation}
\label{star}
z = -(4 \pi h)^{-(d-1)/2}\int_{\R^{d-1}}\exp\left(-\frac{|X|^2}{4h}\right)g(X)dX.
\end{equation}
Now expand the function $g$. We know that $g \in C^{2,\alpha}(\R^{d-1})$, with $g(0)=\partial_{x_i} g(0)=0$. Then, by Taylor's theorem, $g$ can be expressed as 
\begin{align*}
g(X) = \sum_{|\beta|=2} \frac{2}{\beta !} \int_0^1 (1-t) D^{\beta} g(tX) dt X^{\beta}.
\end{align*}
Using the $\alpha$-H{\"o}lder continuity of these second derivatives we show that the difference to the term for $D^{\beta}g(0)$ is of higher order. To this end we add and subtract the term for $D^{\beta}g(0)$ and obtain
\begin{align*}
g(X) =& \sum_{|\beta|=2} \frac{2}{\beta !} \int_0^1 (1-t) D^{\beta} g(0) dt X^{\beta}\\
&+ \sum_{|\beta|=2} \frac{2}{\beta !} \int_0^1 (1-t) (D^{\beta} g(tX) - D^{\beta}g(0)) dt X^{\beta}\\
=& \sum_{|\beta|=2} \frac{1}{\beta !} D^{\beta} g(0) X^{\beta} + \sum_{|\beta|=2} \frac{2}{\beta !} \int_0^1 (1-t) (D^{\beta} g(tX) - D^{\beta} g(0)) dt X^{\beta}.
\end{align*}
We regard the contribution to \eqref{star} of the summands of the first term. Due to our basis diagonalizing $D^2g(0)$, for $\beta=e_i+e_j$ with $i \neq j$, we have $D^{\beta}g(0)=0$. By a simple substitution argument, we notice that multiplying the kernel with $X_i^a$ for $a>0$ increases the order in $h$ of the integral by $a/2$. We also notice
\begin{align*}
(4 \pi h)^{-(d-1)/2} \int_{\R^{d-1}} \exp \left( -\frac{|X|^2}{4h} \right) \left( \frac{X_i}{\sqrt{h}} \right)^2 dX = 2.
\end{align*}
Indeed,
\begin{align*}
-(4 \pi h)^{-(d-1)/2}\int_{\R^{d-1}}\exp\left(-\frac{|X|^2}{4h}\right) \sum_{i=1}^{d-1} \frac{1}{2} \partial_i^2 g(0) X_i^2 dX &= -h \sum_{i=1}^{d-1} \partial_i^2 g(0)\\
&= -hH(x_0;E_0).
\end{align*}
We use H\"older continuity to show that the second term is of higher order in $h$:
\begin{align*}
&\left| \sum_{|\beta|=2} \frac{2}{\beta !} \int_0^1 (1-t) (D^{\beta} g(tX) - D^{\beta} g(0)) dt X^{\beta} \right|\\
&\leq \sum_{|\beta|=2} \frac{2}{\beta !} \int_0^1 (1-t) |D^{\beta} g(tX) - D^{\beta} g(0)| dt |X|^{\beta}\\
&\leq C \sum_{|\beta|=2} \frac{2}{\beta !} \int_0^1 (1-t)t^{\alpha} dt |X|^{\beta} |X|^{\alpha}\\
&\leq C' |X|^{2 + \alpha}.
\end{align*}
Multiplying this with the heat kernel and integrating we see that the contribution to \eqref{star} can be estimated by
\begin{align*}
\left| (4 \pi h)^{-(d-1)/2}\int_{\R^{d-1}} \exp \left( -\frac{|X|^2}{4h} \right)C'|X|^{2+\alpha} dX \right| = \mathcal{O}(h^{1+ \alpha / 2})\text{.}
\end{align*}
It remains to show that the other summands of the power series are of higher order in $h$. We remember to multiply $-\sqrt{4 \pi h}$ and take the absolute value for
\begin{align}
\notag&\left|(4 \pi h)^{-(d-1)/2}\int_{\R^{d-1}} \exp\left(-\frac{|X|^2}{4h}\right)\sqrt{4h}\sum_{k=1}^{\infty}\frac{(-1)^k}{(2k+1)k!}\left(\frac{g(X)+z}{\sqrt{4h}}\right)^{2k+1}dX \right| \\
&\quad\leq h^{-(d-1)/2} \int_{\R^{d-1}}\exp\left(-\frac{|X|^2}{4h}\right)\sum_{k=1}^{\infty}\sum_{l=0}^{2k+1} C(k,l) \frac{|g(X)|^l |z|^{2k+1-l}}{h^k}dX.\label{starstar}
\end{align}
By the Taylor expansion of $g$, we can bound $|g|^l$ by a polynomial, for which $2l$ is the lowest degree of a non-zero monomial. By monotone convergence we interchange limits for
\begin{align*}
&h^{-(d-1)/2} \sum_{k=1}^{\infty}\sum_{l=0}^{2k+1} C(k,l) \int_{\R^{d-1}}\exp\left(-\frac{|X|^2}{4h}\right) \frac{|g(X)|^l |z|^{2k+1-l}}{h^k}dX \\
\leq &\sum_{k=1}^{\infty} \sum_{l=0}^{2k+1} C'(k,l) h^{l-k} |z|^{2k+1-l}.
\end{align*}
This term is of order $h^{3/2}$ if $z = \Oc (h^{5/6})$. This will be given by the following lemma, the proof following immediatly after this proof.

\begin{lemma}
\label{z:bound}
Let $E_0$ be a bounded set with $C^2$-boundary. Then $z(x) = \Oc(h^{1-\eps})$ for all $\eps >0$ uniformly for all boundary points $x$.
\end{lemma}

Thus the first part of the proof is finished and
\begin{align*}
z(x) = -hH(x;E_0) + \Oc(h^{1 + \alpha/2}).
\end{align*}

Let us now assume that $\partial E_0$ is of class $C^4$. Then in particular, $z(x) = \Oc(h)$. Thus all summands in \eqref{starstar} of the power series are at least of order $h^2$. In addition, the Taylor series of $g$ is of the form
\begin{align*}
g(X) = \frac{1}{2}\sum_{i,j=1}^{d-1}D_{ij}g(0) X_i X_j + \frac{1}{6}\sum_{i,j,k=1}^{d-1}D_{ijk}g(0) X_i X_j X_k + \sum_{|\beta|=4}R_{\beta}(X)X^{\beta}
\end{align*}
for some functions $R_{\beta}$, bounded in $\|g\|_{C^4}$. The first term is exactly as in the last proof. The third term has as additional factor $X^{\beta}$, with $|\beta|=4$ thus is of order $h^2$. The second term is an odd function, thus for all $i,j,k$
\begin{align*}
\int_{\R^{d-1}} \exp\left(-\frac{|X|^2}{4h}\right)D_{ijk}g(0) X_i X_j X_k dX = 0.
\end{align*}
Thus, all other terms are at least of order $h^2$ and
\begin{align*}
z(x) = -hH(x;E_0) + \Oc(h^{2}). & \qedhere
\end{align*}
\end{proof}

\begin{proof}[Proof of Lemma~\ref{z:bound}]
Let $E_0$ be a bounded set in $\R^{d}$ with $C^2$-boundary. Then there is a $R>0$ such that for every $x \in \partial E_0$ and
\begin{align*}
\{x\} = \overline{B_R(p)} \cap \partial E_0 \text{ for } p = x \pm R\nu(x).
\end{align*}
Let w.l.o.g.\ $B_R(y) \subs E_0$. We fix the coordinate system such that $\nu(x) = e_d$ and $x-2c_1h^{1-\eps}e_d = 0$ for some $c_1 >0$. We then find an orthotope $A_h \subs B_R(y)$, which will allow us to separate dimensions in the convolution integral,
\begin{align*}
A_h \coloneqq  (-c_2h^{(1-\eps)/2}, c_2h^{(1-\eps)/2})^{d-1} \times (-R,c_1h^{1-\eps}).
\end{align*}
The side length in "horizontal" directions comes from Pythagoras' Theorem and the constant ratio between diameter and side length of a $(d-1)$-dimensional hyper cube. We thus find a constant $c_2>0$ with
\begin{align*}
c_2 h^{(1-\eps)/2} \leq  \frac{\sqrt{R^2 - (R - c_1h^{1-\eps})^2}}{\sqrt{d-1}} = \sqrt{\frac{2 c_1 R}{d-1}} \, h^{(1-\eps)/2} - \Oc(h^{1-\eps})
\end{align*}
for $h$ small enough. This choice of $c_2$ implies that $A_h \subs B_R(y)$. See Figure~\ref{fig:orthotope} for a graphical representation.\\

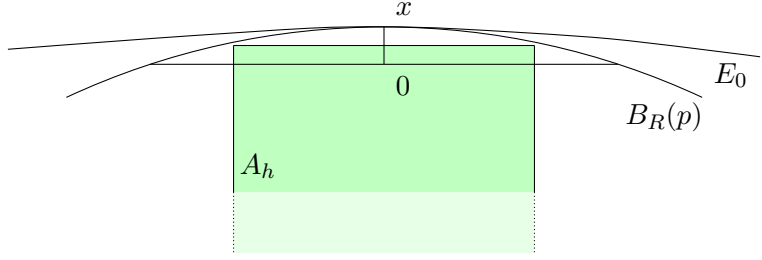
\begin{figure}
\centering
\begin{tikzpicture}

\draw plot [smooth] coordinates {(-5,-0.3) (-3,-.15) (-1,-0.02) (0,0) (1,-0.03) (3,-0.16) (5,-0.4)};

\draw (4.23, -0.94) arc (65:115:10);

\fill [green!25] (-2,-0.25) rectangle (2,-2.2);

\fill [green!10] (-2,-2.2) rectangle (2,-3);

\draw (0,0) -- (0,-.5);

\draw (-3.11,-.5) -- (3.11,-.5);

\draw (-2,-0.25) -- (2,-0.25);

\draw (-2,-0.25) -- (-2,-2.2);

\draw (2,-0.25) -- (2,-2.2);

\draw [densely dotted] (-2,-2.2) -- (-2,-3);

\draw [densely dotted] (2,-2.2) -- (2,-3);

\node [above right=1pt of {(0,0)}] {$x$};

\node [below right=1pt of {(0,-.5)}] {$0$};

\node [below right=2pt of {(4.2,-.3)}] {$E_0$};

\node [below right=3pt of {(3,-.8)}] {$B_R(p)$};

\node [above right=2pt of {(-2.1,-2.2)}] {$A_h$};
\end{tikzpicture}
\caption{Construction of the orthotope $A_h$}
\label{fig:orthotope}
\end{figure}

We show that for $h$ small enough $0$ is in $E_1^h$. That is, the boundary doesn't move this far. Since $A_h \subs B_R(p) \subs E_0$, 
\begin{align*}
(G_h \ast \chi_{E_0})(0) \geq& \; (G_h \ast \chi_{A_h})(0)\\
=& \; \int_{-R}^{c_1 h^{1-\eps}} (4 \pi h)^{-1/2} \exp\left(-\frac{s^2}{4h}\right) ds\\
&\times \left[ \int_{-c_2 h^{(1-\eps)/2}}^{c_2 h^{(1-\eps)/2}} (4 \pi h)^{-1/2} \exp \left(-\frac{s^2}{4h}\right) ds \right]^{d-1}.
\end{align*}
This can be expressed in error functions. We calculate
\begin{align*}
\int_0^r\exp(-cs^2) ds = \frac{\sqrt{\pi}}{2 \sqrt{c}} \erf(\sqrt{c}r)
\end{align*}
and set $c = (4h)^{-1/2}$ to obtain
\begin{align*}
(G_h \ast \chi_{E_0})(0) \geq \frac{1}{2} \left[ \erf \left( \frac{c_1}{2}h^{1/2-\eps}\right) + \erf \left( \frac{R}{2} h^{-1/2} \right) \right] \left[ \erf \left( \frac{c_2}{2} h^{-\eps/2} \right) \right]^{d-1}.
\end{align*}
If the right-hand side is larger than $1/2$, then $0$ is in $E_1^h$. That is equivalent to 
\begin{equation}
\label{conv:erf}
1 > \frac{\left[ \erf \left( \frac{c_2}{2}h^{-\eps/2} \right) \right]^{1-d} - \erf \left( \frac{R}{2} h^{-1/2} \right)}{\erf \left( \frac{c_1}{2}h^{1/2-\eps}\right)}.
\end{equation}
As we are interested in the limit $h \downarrow 0$, we apply l'H\^{o}pital's rule. Numerator and denominator are differentiable and converge to zero when $h$ converges to zero. For $h$ positive, the denominator is positive. We differentiate both terms with respect to $h$:
\begin{align*}
\left( \left[ \erf \left( \frac{c_2}{2} h^{-\eps/2} \right) \right]^{1-d} \right)' =& \; \frac{(d-1) c_2 \eps}{2 \sqrt{\pi}} \left[ \erf \left( \frac{c_2}{2} h^{-\eps/2} \right) \right]^{-d} 
\\&\times h^{-1-\eps/2} \exp \left( -\frac{c_2^2}{4} h^{-\eps} \right), \\
\left( - \erf \left( \frac{R}{2} h^{-1/2} \right) \right)' =& \; \frac{R}{2 \sqrt{\pi}} h^{-3/2} \exp \left( -\frac{R^2}{4} h^{-1} \right), \\
\left( \erf \left( \frac{c_1}{2} h^{1/2-\eps} \right) \right)' =& \; \frac{c_1}{\sqrt{\pi}} \left( \frac{1}{2}-\eps \right) h^{-1/2-\eps} \exp \left( -\frac{c_1^2}{4} h^{1-2\eps} \right).
\end{align*}
The last term being positive, we regard the fraction of derivatives
\begin{align*}
&\frac{ \frac{(d-1) c_2 \eps}{2 \sqrt{\pi}} \left[ \erf \left( \frac{c_2}{2} h^{-\eps/2} \right) \right]^{-d} h^{-1-\eps/2} \exp \left( -\frac{c_2^2}{4} h^{-\eps} \right) + \frac{R}{2 \sqrt{\pi}} h^{-3/2} \exp \left( -\frac{R^2}{4} h^{-1} \right) }{ \frac{c_1}{\sqrt{\pi}} \left(\frac{1}{2}-\eps \right) h^{-1/2-\eps} \exp \left( -\frac{c_1^2}{4} h^{1-2\eps} \right) } \\
&= \frac{(d-1) c_2 \eps}{2 c_1 (1/2 - \eps)} \left[ \erf \left( \frac{c_2}{2} h^{-\eps/2} \right) \right]^{-d} h^{(\eps - 1)/2} \exp \left( \frac{c_1^2}{4} h^{1/2 - \eps} - \frac{c_2^2}{4} h^{-\eps} \right) \\
&\quad + \frac{R}{2c_1 (1/2 -\eps)} h^{\eps-1} \exp \left( \frac{c_1^2}{4} h^{1/2 - \eps} - \frac{R^2}{4} h^{-1} \right),
\end{align*}
which converges to zero as $h$ converges to zero. Then by l'H\^{o}pital's rule the right-hand side of \eqref{conv:erf} also converges to zero and the inequality is true for $h$ small enough.
\end{proof}

\begin{proof}[Proof of Proposition~\ref{anis:per}]
Since $E$ is a set of finite perimeter, $\chi_E$ and $\chi_{E^c}$ are functions of bounded variation on $\{0, 1\}$. Then from \cite{Esedoglu:Otto} Lemma A.3 we recall
\begin{align*}
\underset{h \downarrow 0}{\lim} \frac{1}{\sqrt{h}} \int_{\Rd} \chi_E(x-\sqrt{h}y) \chi_{E^c}(x)dx = \int_{\partial^* E} (\nu(x) \cdot y)_+ d\mathcal{H}^{d-1}(x).
\end{align*}
as well as
\begin{align*}
P_{K,h}(E) &= \frac{1}{\sqrt{h}} \int_{\Rd} \int_{\Rd} K_h(x-y) \chi_{E}(y) \chi_{E^c}(x) dx dy\\
&= \int_{\Rd} K_h(y) \frac{1}{\sqrt{h}} \int_{\Rd} \chi_E(x-y) \chi_{E^c}(x) dx dy\\
&= \int_{\Rd} K(y) \frac{1}{\sqrt{h}} \int_{\Rd} \chi_E(x - \sqrt{h}y) \chi_{E^c}(x) dx dy.
\end{align*}
From (A.17) in \cite{Esedoglu:Otto} we obtain
\begin{align}
\frac{1}{\sqrt{h}} \int_{\Rd} \chi_E(x - \sqrt{h}y) \chi_{E^c}(x) dx &\leq \frac{1}{\sqrt{h}} \int_{\Rd} |\chi_E(x - \sqrt{h}y) - \chi_E(x)| dx \notag\\
\label{eq:weak perimeter bound}&\leq |y|P(E).
\end{align}
By the bound \eqref{K:bound}, $|y|K(y) \in L^1$, and by assumption, $E$ is a set of finite perimeter. Thus we can apply dominated convergence and due to the symmetry of $K$ we obtain
\begin{align*}
\underset{h \downarrow 0}{\lim} P_{K,h}(E) = \frac{1}{2} \int_{\R^d} K(y) \int_{\partial^* E} |\nu(x) \cdot y| d\mathcal{H}^{d-1}(x) dy.
\end{align*}
From Fubini's Theorem we obtain
\begin{align*}
\frac{1}{2} \int_{\partial^* E} \int_{\R^d} |\nu(x) \cdot y| K(y) dy\, d\mathcal{H}^{d-1}(x) = P_{\sigma_K}(E)
\end{align*}
which concludes the proof.
\end{proof}

In the following proof we will use the non-negativity of $K$ only in the discussion of a priori bounds of z. This allows us to quote major parts of the proof in the proof of Theorem~\ref{negative:Ker}.

\begin{proof}[Proof of Theorem~\ref{anis:main} in case of (I)]
We assume $K \in C_c^{0,\alpha}(\R^d)$ and again fix $x_0 \in \partial E_0$. As the boundary of $E_0$ is of the class $C^{2,\alpha}$ we find a neighborhood U of $x_0$, a coordinate system with $x_0=0$ and $e_d =\nu(x_0)$ and a function $g \in C_c^{2,\alpha}(\R^{d-1})$, such that $g(0) = \partial_{x_i} g(0) = 0$ and in this coordinate system
\begin{align*}
E_0 \cap U = \{(X,y) \in \R^{d-1}\times \R | y < g(X)\} \cap U.
\end{align*}
We further choose $g$ to have compact support in some large ball. By the spectral theorem, the symmetric matrix $D^2g(0)$ is diagonalized by an orthonormal basis of eigenvectors. We fix these as $(e_i)_{i=0}^{d-1}$ in the coordinate system and define the function 
\begin{align*}
F(X,y,h) \coloneqq (K_h \ast \chi_{E_0})(X,y).
\end{align*}
We then solve $F(0,z,h)=1/2$ for $z \in \R$. Note that the signed distance $z(x_0)$ is not necessarily the only solution, since the $(1/2)$-level set may have positive measure. We show the claim to hold for all solutions.\\

Since $K$ is compactly supported, there is a positive $h_0$ such that $\supp(K_h) \subs U$  for all $h \in (0,h_0)$. We choose such an $h$. Then
\begin{align*}
F(0,z,h) &= (K_h \ast \chi_{E_0 \cap U})(0,z)\\
&= h^{-d/2} \int_{\R^{d-1}} \int_{-\infty}^{g(X)} K(X/\sqrt{h}, (y-z)/\sqrt{h}) dy dX
\end{align*}
and, by substitution of $y$ with $y+z$,
\begin{align*}
= h^{-d/2} \int_{\R^{d-1}} \int_{-\infty}^{g(X) + z} K(X/\sqrt{h}, y/\sqrt{h}) dy dX.
\end{align*}
As $K$ is even and $\int_{\Rd}K=1$, integral of $K$ over any half space is $1/2$. Remembering that $g(X)+z$ may be negative we write
\begin{align*}
F(0,z,h) = \frac{1}{2} + h^{-d/2} \int_{\R^{d-1}} \int_0^{g(X)+z} K(X/\sqrt{h}, y/\sqrt{h}) dy dX.
\end{align*}
Thus $z$ solves $F(0,z,h)=1/2$ if and only if
\begin{align*}
0 = h^{-d/2} \int_{\R^{d-1}} \int_0^{g(X)+z} K(X/\sqrt{h}, y/\sqrt{h}) dy dX.
\end{align*}

In the rest of this proof, we will show that the solution $z$ satisfies \eqref{eq:initial movement PDE}. 
We begin by giving some motivation. Using that the domain of integration lies close to the hyperplane $x_d=0$ on the support of $K_h$, we approximate
\begin{align*}
 0 &= h^{-d/2} \int_{\R^{d-1}} \int_0^{g(X)+z} K(X/\sqrt{h}, y/\sqrt{h}) dy dX\\
&\approx h^{-d/2} \int_{\R^{d-1}} \int_0^{\frac{1}{2} \sum \partial_i^2g(0)X_i^2 + z} K(X/\sqrt{h},0) dy dX\\
&= h^{-d/2} \int_{\R^{d-1}} K(X/\sqrt{h},0) \left(\frac{1}{2} \sum \partial_i^2g(0)X_i^2 + z\right) dX\\
&= \frac{1}{2} h^{1/2} \sum \partial_i^2g(0) \int_{\R^{d-1}} K(X,0) X_i^2 dX + h^{-1/2} \int_{\R^{d-1}} K(X,0) dX z\\
&= \frac{1}{2} h^{1/2} H_{\sigma}(x_0,E_0) + h^{-1/2} \frac{1}{2\mu_K(e_d)} z(x_0).
\end{align*}

Now we make this computation (meaning the step ``$\approx$'') rigorous. We start by writing
\begin{align*}
0 =& \; h^{-d/2} \int_{\R^{d-1}} \int_0^{g(X)+z} K(X/\sqrt{h}, y/\sqrt{h}) dy dX\\
=& \; h^{-d/2} \int_{\R^{d-1}} \int_0^{g(X)+z} K(X/\sqrt{h}, 0) dy dX \\
&+ h^{-d/2} \int_{\R^{d-1}} \int_0^{g(X)+z} ( K(X/\sqrt{h}, y/\sqrt{h}) - K(X/\sqrt{h}, 0) ) dy dX.
\end{align*}
We denote by $\Err_h(z)$ the second term on the right-hand side. Commuting the trivial $y$-integral in the first term, we have 
\begin{align*}
0 = h^{-d/2} \int_{\R^{d-1}} K(X/\sqrt{h}, 0) (g(X) + z) dX + \Err_h(z).
\end{align*}
Shifting the $z$-term to the left-hand side, we substitute
\begin{align*}
h^{-1/2} \int_{\R^{d-1}} K(X,0) dX z = -h^{-d/2} \int_{\R^{d-1}} K(X/\sqrt{h}, 0) g(X) dX - \Err_h(z).
\end{align*}
But now the integral on the left-hand side is exactly the mobility term $1/(2\mu_K(e_d))$, and
\begin{equation}
\label{plus}
\frac{1}{\mu_K(e_d)} z = -h^{-(d-1)/2} 2\int_{\R^{d-1}} K(X/\sqrt{h}, 0) g(X) dX - 2\sqrt{h} \Err_h(z).
\end{equation}

Next, we observe that $g$ is in $C^{2,\alpha}$ and $g(0)=\partial_ig(0)=0$. Thus, by Taylor's theorem
\begin{align*}
g(X) = \sum_{|\beta|=2} \frac{2}{\beta!} \int_0^1 (1-t) D^{\beta} g(tX) dt \, X^{\beta}
\end{align*}
and $D^{\beta}g$ is $\alpha$-H\"older continuous. Again we split this term, showing the difference term to be of higher order:
\begin{align*}
g(X) = &\sum_{|\beta|=2} \frac{2}{\beta!} \int_0^1 (1-t) D^{\beta} g(0) dt \, X^{\beta}\\ 
&+ \sum_{|\beta|=2} \frac{2}{\beta!} \int_0^1 (1-t) ( D^{\beta} g(tX) - D^{\beta} g(0) ) dt \, X^{\beta}.
\end{align*}
The first term generates the anisotropic curvature of $\partial E_0$ in $x_0$. Indeed, due to our choice of coordinate system, $\partial_i\partial_jg(0)=0$ for $i \neq j$, and
\begin{align*}
-h^{-(d-1)/2} 2\int_{\R^{d-1}} &K(X/\sqrt{h}, 0) \sum_{|\beta|=2} \frac{2}{\beta!} \int_0^1 (1-t) D^{\beta} g(0) dt \, X^{\beta} dX\\
= &-h^{-(d-1)/2} \sum_{i=1}^{d-1}\partial_i^2g(0) \int_{\R^{d-1}} K(X/\sqrt{h},0) X_i^2 dX
\end{align*}
Substituting $X$ for $X/\sqrt{h}$, due to the factor $X_i^2$, we get
\begin{align*}
= -h \sum_{i=1}^{d-1} \partial_i^2g(0) \int_{\R^{d-1}} K(X,0) X_i^2 dX.
\end{align*}
But the second partial derivatives of $g$ in $0$ are exactly the respective (isotropic) principal curvatures $\kappa_i$. If $K$ satisfies the conditions of Proposition~\ref{div:sigma}, by \eqref{anisotropic:curvature:expression} this is 
\begin{align*}
-h H_{\sigma}(x_0;E_0).
\end{align*}
We need to show, that with $f(X) = (X \cdot MX) + a$ for any symmetric $(d-1)$-matrix $M$ and number $a$
\begin{align*}
\underset{\eps \rightarrow 0}{\lim} \underset{O \in \mathrm{O}(d)}{\sup} \int_{\R^{d-1}} |K \circ O(X,\eps f(X)) - K \circ O(X,0)| (1 + |X|^2) dX = 0.
\end{align*}
Since $K$ is compactly supported and H\"older continuous, we can estimate the left-hand side by
\begin{align*}
\underset{\eps \rightarrow 0}{\lim} \eps^{\alpha} C \int_{B_R^{d-1}(0)} |f(X)|^{\alpha} (1 + |X|^2) dX = 0,  
\end{align*}
and the condition is satisfied.

It remains to show that the contribution of the two error terms is of higher order. We begin with the term concerning $g$. Since $D^{\beta}g$ is $\alpha$-H\"older continuous, 
\begin{align*}
&\left| \sum_{|\beta|=2} \frac{2}{\beta!} \int_0^1 (1-t) (D^{\beta} g(tX) - D^{\beta}g(0)) dt \, X^{\beta} \right| \\
&\leq \sum_{|\beta|=2} \frac{2}{\beta!} \int_0^1 (1-t) |D^{\beta} g(tX) - D^{\beta} g(0)| dt \, |X^{\beta}|\\
&\leq \sum_{|\beta|=2} \frac{2}{\beta!} \int_0^1 (1-t) t^{\alpha} dt \, |X^{\beta}||X|^{\alpha}\\
&\leq C |X|^{2+\alpha}
\end{align*}
for some $C>0$. We now find a bound for the error, resulting from replacing $g(X)$ in \eqref{plus} by its second order Taylor polynomial:
\begin{align*}
\left| 2 h^{-(d-1)/2} \int_{\R^{d-1}} \right. &K(X/\sqrt{h},0) \\
&\left. \times \sum_{|\beta|=2} \frac{2}{\beta!} \int_0^1 (1-t) (D^{\beta} g(tX) - D^{\beta}g(0)) dt \, X^{\beta} dX \right|\\
\leq \, C h^{-(d-1)/2} \int_{\R^{d-1}} &|K(X/\sqrt{h},0)| |X|^{2+\alpha} dX.
\end{align*}
We substitute $X$ for $X/\sqrt{h}$ in order to rewrite the right-hand side as
\begin{align*}
C h^{(1+\alpha)/2}  \int_{\R^{d-1}} |K(X,0)| |X|^{2+\alpha} dX.
\end{align*}
The integral is finite because $K$ has compact support. Hence, this term is of order $h^{(1+\alpha)/2}$.\\

We now show the smallness of the other error term, namely, $\sqrt{h}\Err_h(z) = \Oc(h^{1+\alpha/2})$. Recall that $K$ is $\alpha$-H\"older continuous, i.e. 
\begin{align*}
|K(x)-K(x')| \leq C|x-x'|^{\alpha}
\end{align*}
and, since K has compact support, there is a number $R>0$ such that 
\begin{align*}
|K(x)-K(x')| \leq C|x-x'|^{\alpha}\chi_{\{\min(|x|,|x'|)<R\}}(x,x').
\end{align*}
We apply this to $\sqrt{h}\Err_h(z)$ and get
\begin{align*}
\Big| &\sqrt{h} \Err_h(z) \Big| 
\\ &\leq h^{-(d-1)/2} \int_{\R^{d-1}} \int_{-|g(X)|-|z|}^{|g(X)|+|z|} | K(X/\sqrt{h}, y/\sqrt{h}) - K(X/\sqrt{h}, 0)| dy\, dX \\
&\leq h^{-(d-1)/2} \int_{B_{R\sqrt{h}}^{d-1}(0)} \int_{-|g(X)|-|z|}^{|g(X)|+|z|} C |y/\sqrt{h}|^{\alpha} dy\, dX \\
&\leq h^{-(d-1)/2} \int_{B_{R\sqrt{h}}^{d-1}(0)} \frac{2C}{1+\alpha} h^{-\alpha/2} (|g(X)|+|z|)^{1+\alpha}dX.
\end{align*}
As $(\cdot)^{1+\alpha}$ is convex, we apply Jensen's inequality:
\begin{align*}
\left| \sqrt{h} \Err_h(z) \right| \leq h^{-(d-1+\alpha)/2} \int_{B_{R\sqrt{h}}^{d-1}(0)} C'(|g(X)|^{1+\alpha}+|z|^{1+\alpha}) dX.
\end{align*}
Using the Taylor rest representation of $g$ on a closed ball and substitution we bound the term for the first summand
\begin{align*}
&\leq h^{-(d-1+\alpha)/2} \int_{B_{R\sqrt{h}}^{d-1}(0)} C' |g(X)|^{1+\alpha} dX\\
&\leq h^{-(d-1+\alpha)/2} \int_{B_{R\sqrt{h}}^{d-1}(0)} C''|X|^{2(1+\alpha)} dX\\
&\leq C''' h^{1+\alpha/2}.
\end{align*}

It remains to show that $z$ is of order $h$ or the mobility term $1/\mu_K(e_d) = 0$. We remember that
\begin{align*}
\frac{1}{\mu_K(e_d)} = 2\int_{\R^{d-1}} K(X,0) dX.
\end{align*}
Thus if $1/\mu_K(e_d) > 0$, there is a domain in $\supp(K) \cap \{x \, | \, x_d = 0\}$, where $K$ is positive. Then, as $K$ is continuous, for all $\eps>0$
\begin{align*}
\int_{\R^{d-1}} \int_0^{\eps} K(X,y) dy dX > 0.
\end{align*}
Since the boundary of $E_0$ is compact and of class $C^{2}$, there is a number $R>0$, such that for every $x \in \partial E_0$ we have
\begin{align*}
\{x\} = \overline{B_R(y)} \cap \partial E_0 \text{ for }p = x \pm R\nu(x).
\end{align*}
Without loss of generality, we fix some pair $(x,p)$, with $p \in E_0$, and a coordinate system, such that $\nu(x)=e_d$ and $x-2che_d=0$ for some $c$ with
\begin{align*}
c > \frac{L^2}{2R},
\end{align*}
where $L$ is such that $B_L(0)\sups \supp(K)$. Then, for $h$ small enough, 
\begin{align*}
B_{L\sqrt{h}}(0)\cap \{y < ch\} \subs B_R(y) \subs E_0.
\end{align*}
For a visualization see Figure~\ref{fig:dome}. $K$ being an even function with support in $B_L(0)$, and positive integral on $B_L(0)\cap \{0 < y < ch\}$, we have
\begin{align*}
(K_h \ast \chi_{E_0})(0) \geq \int_{B_{L\sqrt{h}}(0) \cap \{y < ch\}} K_h(x) dx = \int_{B_L(0) \cap \{y < c\sqrt{h}\}} K(x) dx > 1/2
\end{align*}
and $0$ is not in $E_1$ for any $h$ small enough. Thus $z \in \mathcal{O}(h)$ and
\begin{align*}
&h^{-(d+\alpha)/2} \int_{B_{R\sqrt{h}}^{d-1}(0)} C'|z|^{1+\alpha} dX\\
&\leq C'' h^{1+\alpha/2},
\end{align*}
which concludes $\sqrt{h}\Err_h(z) = \Oc(h^{1+\alpha/2})$. If $1/\mu_K(e_d)=0$, the left-hand side of our equation is zero. Furthermore, 
\begin{align*}
H_{\sigma_{K}}(x;E_0) = \sum_{i=1}^{d-1} \kappa_i \int_{\R^{d-1}} X_i^2 K(X,0)dX = 0
\end{align*}
and the equation we want to prove reads $0 = \mathcal{O}(h^{1+\alpha/2})$, which is true. Collecting all terms, we find
\begin{align*}
\frac{1}{\mu_K(e_d)} z(x) = - h H_{\sigma_{K}}(x;E_0) + \mathcal{O}(h^{1+\alpha/2}),
\end{align*}
which concludes the proof in the first case.
\end{proof}

\begin{figure}
\centering
\begin{tikzpicture}

\fill [green!25] (0,-.5) circle (2.2);

\fill [white] (-5,1.7) rectangle (5,-.25);

\draw plot [smooth] coordinates {(-5,-0.3) (-3,-.15) (-1,-0.02) (0,0) (1,-0.03) (3,-0.16) (5,-0.4)};

\draw (4.23, -0.94) arc (65:115:10);

\draw (0,0) -- (0,-.5);

\draw (-3.11,-.5) -- (3.11,-.5);

\draw (-2.2,-0.25) -- (2.2,-0.25);

\draw (-2.18,-0.25) arc (173:367:2.2);

\node [above right=1pt of {(0,0)}] {$x$};

\node [below right=1pt of {(0,-.5)}] {$0$};

\node [below right =2pt of {(4.2,-.3)}] {$E_0$};

\node [below right=3pt of {(3,-.8)}] {$B_R(p)$};

\node [below right=1pt of {(-1.9,-1.1)}] {$B_{\sqrt{h}L}(0) \cap \{y<ch\}$};
\end{tikzpicture}
\caption{Dome covering more than half of the support of $K_h$}
\label{fig:dome}
\end{figure}
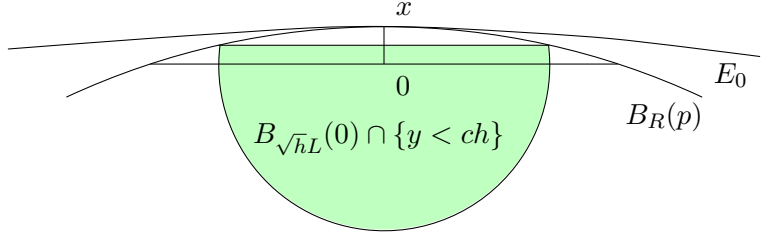

We turn to the second statement and functions with non-compact support, where we have to additionally study decay rates at infinity and integrability.
\begin{proof}[Proof of Theorem~\ref{anis:main} in case of (II)]
Let $K$ satisfy the conditions in \textit{(II)}, then in particular for some $\eps > 0$
\begin{align*}
|K(x)| \leq C|x|^{-d-2-\eps} \quad \text{for all } x\in \R^d.
\end{align*}
For some point $x_{0} \in \partial E_0$ with a neighborhood $U$ we set our coordinate system and the boundary function $g$ as in the proof in the case of \textit{(I)} above. In particular $x_0=0$. Let $R$ be a positive number with $B_R(0) \subs U$. Then we find a higher order bound for the difference between the global and the localized problem
\begin{align*}
\left| (K_h \ast \chi_{E_0})(0,z) - \int_{\R^{d-1}} \int_{-\infty}^{g(X)} K_h(X,y-z) dy \, dX \right| \leq \int_{B_R^c(0)} |K_h(x)| dx.
\end{align*}
Using the asymptotic bound of $K$ we obtain
\begin{align*}
\int_{B_R^c(0)} |K_h(x)| dx &\leq h^{-d/2} \int_{B_R^c(0)} C|x/\sqrt{h}|^{-d-2-\eps} dx\\
&= h^{1+\eps/2} C\int_{B_R^c(0)} |x|^{-d-2-\eps} dx
\end{align*}
and, since the integral is finite for any $h$,
\begin{align*}
\int_{B_R^c(0)} |K_h(x)| dx \leq C' h^{1+\eps/2}.
\end{align*}
Fully analogously to \textit{(I)}, we obtain that our equation $F(0,z,h)=1/2$ is equivalent to
\begin{align*}
0 = h^{-d/2} \int_{\R^{d-1}} \int_0^{g(X) + z} K(X/\sqrt{h},y/\sqrt{h}) dy \, dX + \Oc(h^{1+\eps/2}).
\end{align*}
Further fully analogously, this is equivalent to
\begin{align*}
\frac{1}{\mu_K(e_d)} z = &- h^{-(d-1)/2} 2\int_{\R^{d-1}} K(X/\sqrt{h},0)g(X) dX\\ 
&- h^{-(d-1)/2} \int_{\R^{d-1}} 2\int_0^{g(X)+z} (K(X/\sqrt{h},y/\sqrt{h}) - K(X/\sqrt{h},0)) dy dX\\
&+ \Oc(h^{3/2+\eps/2}).
\end{align*}
We again want to show the second summand on the right-hand side to be of higher order in $h$. Due to the weaker condition on the kernel $K$, we will get less a-priori information on the order of $z$. Instead, we have a decaying Lipschitz bound on $K$:
\begin{align*}
|K(X/\sqrt{h},y/\sqrt{h}) - K(X/\sqrt{h},0)| \leq C \left| \frac{X}{\sqrt{h}} \right|^{-d-3-\delta} \left| \frac{y}{\sqrt{h}} \right|.
\end{align*}
First, we estimate
\begin{align*}
&\left| h^{-(d-1)/2} \int_{\R^{d-1}} \int_0^{g(X)+z} (K(X/\sqrt{h},y/\sqrt{h}) - K(X/\sqrt{h},0)) dy dX \right| \\
&\leq h^{-(d-1)/2} \int_{\R^{d-1}} \int_{-|g(X)|-|z|}^{|g(X)|+|z|} |K(X/\sqrt{h},y/\sqrt{h}) - K(X/\sqrt{h},0)| dy dX,
\end{align*}
and then use both the general Lipschitz bound on $K$, as well as the decaying Lipschitz bound to obtain
\begin{align*}
\leq& Ch^{-(d-1)/2} \int_{\R^{d-1}} \min\left(\Lip(K), \left( \frac{|X|}{\sqrt{h}} \right)^{-d-3-\delta} \right) \frac{(|g(X)| + |z|)^2}{\sqrt{h}} dX\\
\leq& Ch^{3/2+\delta/2} \int_{\R^{d-1} \cap \{|X|>\Lip(K)^{-1/(d+2+\delta)}\sqrt{h}\}} |X|^{-d-2-\delta} (|g(X)|+|z|)^2 dX\\
 &+ Ch^{-d/2}\int_{\R^{d-1} \cap \{|X|<\Lip(K)^{-1/(d+2+\delta)}\sqrt{h}\}} \Lip(K) (|g(X)| + |z|)^2 dX.
\end{align*}
We write $a = \Lip(K)^{-1/(d+2+\delta)}$ and apply Jensen's inequality to get
\begin{align*}
\leq& C'h^{3/2+\delta/2} \int_{(B_{a\sqrt{h}}^{d-1}(0))^c} |X|^{-d-3-\delta} (|g(X)|^2 + |z|^2) dX\\
 &+ C'h^{-d/2} \int_{B_{a\sqrt{h}}^{d-1}(0)} \Lip(K) (|g(X)|^2 + |z|^2) dX.
\end{align*}
We find bounds for all four summands successively. First, remembering the Taylor expansion on a ball, we get
\begin{align*}
&\int_{(B_{a\sqrt{h}}^{d-1}(0))^c} |X|^{-d-3-\delta} |g(X)|^2 dX\\ 
\leq& \int_{B_1^{d-1}(0) \setm B_{a\sqrt{h}}^{d-1}(0)} |X|^{-d-3-\delta} \|D^2g\|_{C^0}^2 |X|^4 dX\\
&+ \int_{(B_{1}^{d-1}(0))^c} \|g\|_{C^0} |X|^{-d-3-\delta} dX\\
=& \, C \int_{a\sqrt{h}}^1 r^{-1-\delta} dr + C \int_1^{\infty} r^{-5-\delta} dr \leq C' h^{-\delta/2} + C'  
\end{align*}
and the term is of order $h^{3/2}$. For the second bound, we first need an a-priori estimate on the order of $z$, which is given by the following lemma, the proof following this proof.

\begin{lemma}
\label{anis:z:bound}
Let $E_0 \subs \R^{d}$ be bounded and of the class $C^2$, and $K$ a suitable convolution kernel satisfying condition \textit{(II)} of Theorem~\ref{anis:main}. Then for $z(\cdot)$ as defined in Theorem~\ref{anis:main} there is an $\eps'>0$, such that $z(\cdot) \in \Oc(h^{3/4+\eps'})$.
\end{lemma}

Thus $z \in \Oc(h^{3/4+\eps'})$. This implies that
\begin{align*}
\int_{(B_{a\sqrt{h}}^{d-1}(0))^c} |X|^{-d-3-\delta} |z|^2 dX &\leq h^{3/2+2\eps'} \int_{a\sqrt{h}}^{\infty} r^{-5-\delta}\\
&\leq C h^{3/2+2\eps'} h^{-2-\delta/2} = Ch^{2\eps'-1/2-\delta/2}
\end{align*}
and the term is in $\Oc(h^{1+2\eps'})$. For the third summand we observe $g$ on a compact ball. Thus we can use the bound from its Taylor representation to obtain
\begin{align*}
h^{-d/2} \int_{B_{a\sqrt{h}}^{d-1}(0)} \Lip(K) |g(X)|^2 dX &\leq h^{-d/2} \int_{B_{a\sqrt{h}}^{d-1}(0)} \Lip(K) \|g\|_{C^2} |X|^4 dX\\
&\leq h^{-d/2} \Lip(K) \|g\|_{C^2} \int_0^{a\sqrt{h}} r^{d+2} dr\\
&\leq C h^{-d/2} (a\sqrt{h})^{d+3} \leq C' h^{3/2}.
\end{align*}
For the fourth and last term, we again use the a-priori bound of $z$. Then
\begin{align*}
h^{-d/2} \int_{B_{a\sqrt{h}}^{d-1}} \Lip(K) |z|^2 dX &\leq C h^{-1/2} |z|^2\\
&\leq C' h^{1+2\eps'}.
\end{align*}
Thus, this term is of order strictly higher than $1$. After proving the theorem, we know $z \in \Oc(h)$ and $\eps' = 1/4$.\\

To conclude the proof, we regard the term 
\begin{align*}
h^{-(d-1)/2} 2\int_{\R^{d-1}} K(X/\sqrt{h},0)\, g(X) dX.
\end{align*}
The function $g$ being chosen in $C_c^{2,\alpha}$, we can restrict the integral to a large $(d-1)$-dimensional ball. We can show, analogous to the proof of \textit{(I)}, that
\begin{align*}
\left| g(X) - \frac{1}{2} \sum_{i=1}^{d-1} \partial_i^2g(0) X_i^2 \right| \leq C|X|^{2+\alpha}.
\end{align*}
It remains to show, that $|X|^{2+\alpha}$ induces a finite higher order term. By substitution we get
\begin{align*}
h^{-d/2} \int_{\R^{d-1}} |K(X/\sqrt{h},0)| |X|^{2+\alpha} dX \leq h^{1+\alpha/2} \int_{\R^{d-1}} |K(X,0)| |X|^{2+\alpha} dX
\end{align*} 
and using the bound of $K$ we find
\begin{align*}
\int_{\R^{d-1}} |K(X,0)| |X|^{2+\alpha} dX &\leq \int_{\R^{d-1}} \min(\|K\|_{C^0}, |X|^{-d-2-\eps}) |X|^{2+\alpha}dX \\
&\leq C + \int_{(B_a^{d-1}(0))^c} |X|^{\alpha-d-\eps} dX\\
&\leq C + C'\int_a^{\infty} r^{\alpha -2-\eps} dr\\
&\leq C + C''[r^{\alpha-1-\eps}]_a^{\infty},
\end{align*}
which is finite, because $\alpha \leq 1$. Then we have
\begin{align*}
h^{-(d-1)/2} 2\int_{\R^{d-1}} K(X/\sqrt{h},0)\, g(X) dX =& \, h \, \sum_{i=1}^{d-1} \partial_i^2g(0) \int_{\R^{d-1}} K(X,0)X_i^2 dX\\
&+ \Oc(h^{1+\alpha/2}).
\end{align*}
By Proposition~\ref{div:sigma}, this sum is the anisotropic mean curvature of $E_0$ in $x_{0}$, $H_{\sigma}(x_{0};E_0)$ if 
\begin{align*}
\underset{\eps \rightarrow 0}{\lim} \underset{O \in \mathrm{O}(d)}{\sup} \int_{\R^{d-1}} |K \circ O(X,\eps f(X)) - K \circ O(X,0)| (1 + |X|^2) dX = 0,
\end{align*}
with $f(X) = (X \cdot MX) + a$ for any symmetric $(d-1)$-matrix $M$ and number $a$. We use our symmetrically descending Lipschitz bound to get the estimate
\begin{align*}
\underset{\eps \rightarrow 0}{\lim} \eps C \int_{\R^{d-1}} |f(X)| \min\left(\Lip(K), |X|^{-d-3-\delta}\right) (1 + |X|^2) dX.
\end{align*}
It remains to show that this integral is finite. For any $R>0$ the integrand is finite on $B_R(0)$. Outside of $B_R(0)$, we know $|f(X)| \leq C|X|^2$ for some $C<\infty$ and
\begin{align*}
\int_{(B_R^{d-1}(0))^c} |X|^2 |X|^{-d-3-\delta} |X|^2 dX &= \int_{(B_R^{d-1}(0))^c} |X|^{1-d-\delta} dX\\
&= \int_R^{\infty} r^{-1-\delta} dr < \infty.
\end{align*}
Thus the integral is finite, the limit is zero and by Proposition~\ref{div:sigma} and \eqref{anisotropic:curvature:expression}
\begin{align*}
\sum_{i=1}^{d-1} \partial_i^2g(0) \int_{\R^{d-1}} K(X,0)X_i^2 dX = H_{\sigma}(x;E_0).
\end{align*} 
We then collect all terms, to get
\begin{align*}
\frac{1}{\mu_K(e_d)} z(x_0) = -h \, H_{\sigma}(x_0;E_0) + \Oc(h^{1+2\eps'}) + \Oc(h^{3/2}) + \Oc(h^{1+\alpha/2}).
\end{align*}
We set $\tilde{\eps} = \min(2\eps', \alpha/2, 1/2) > 0$. Then
\begin{align*}
\frac{1}{\mu_K(e_d)} z(x_0) = -h \, H_{\sigma}(x_0;E_0) + \Oc(h^{1+\tilde{\eps}}).
\end{align*}
But this implies $z \in \Oc(h)$ and $\eps'=1/4$. We conclude
\begin{align*}
\frac{1}{\mu_K(e_d)} z(x_0) = -h \, H_{\sigma}(x_0;E_0) + \Oc(h^{1+\alpha/2}).& \qedhere
\end{align*}
\end{proof}

\begin{proof}[Proof of Lemma~\ref{anis:z:bound}]
Again, $E_0$ being a bounded set of class $C^2$, there is a number $R>0$, such that for every $x \in \partial E_0$ and
\begin{align*}
p = x \pm R\nu(x)
\end{align*}
we have
\begin{align*}
\{x\} = \overline{B_R(p)} \cap \partial E_0.
\end{align*}
Without loss of generality, we fix some pair $(x,p)$, with $p \in E_0$, a number $\eps'>0$ and a coordinate system, such that $\nu(x) = e_d$ and $x-2ch^{3/4+\eps'}e_d = 0$ for some $c$ with $c>1/2R$. Thus for some $h_0>0$ and all $h \in(0,h_0)$
\begin{align*}
B_{h^{3/8+\eps'/2}}(0) \cap \{y<ch^{3/4+\eps'}\} \subs B_R(p) \subs E_0.
\end{align*}
We then estimate and substitute as follows
\begin{align*}
(K_h \ast \chi_{E_0})(0) &\geq \int_{B_{h^{3/8+\eps'/2}}(0) \cap \{y<ch^{3/4+\eps'}\}} K_h(x) dx\\
&= \int_{B_{h^{-1/8+\eps'/2}}(0) \cap \{y<ch^{1/4+\eps'}\}} K(x) dx.
\end{align*}
With $K$ being even, we then want to show that for all $h$ small enough
\begin{align*}
\int_{(B_{h^{-1/8+\eps'/2}}(0))^c} K(x) dx < \int_{B_{h^{-1/8+\eps'/2}}(0) \cap \{ 0 < y < ch^{1/4+\eps'}\}} K(x) dx,
\end{align*}
which states that the mass of $K_h$ outside a ball much larger than $\sqrt{h}$ is dominated by the mass of $K_h$ in the upper half-space and inside this ball. Since $K_h$ is even and integrates to $1$, any half-space contains exactly the mass $1/2$. The set $E_0$ thus contains strictly more mass than $1/2$. We begin, by finding an upper bound of the left-hand side, using the asymptotic bound of $K$
\begin{align*}
\int_{(B_{h^{-1/8+\eps'/2}}(0))^c} K(x) dx &\leq \int_{(B_{h^{-1/8+\eps'/2}}(0))^c} C|x|^{-d-2-\eps} dx\\
&= C' \int_{h^{-1/8+\eps'/2}}^{\infty} r^{-3-\eps} dr\\
&= C'' h^{(1/8-\eps'/2)(2+\eps)}.
\end{align*}
To find a lower bound of the right-hand side, we need to find a lower bound of $1/\mu_K$. Since we have given that the mobility is positive, we show that $1/\mu_K$ is continuous on the compact set $\Sd$. Let $n,n' \in \Sd$ with $|n-n'|<\delta'$ and $n' = On$. Then we obtain
\begin{align*}
&\left| \frac{1}{\mu_K(n)} - \frac{1}{\mu_K(n')} \right| \\
&\quad= 2 \left| \int_{\{(x \cdot n)=0\}} K(x) d\mathcal{H}^{d-1}(x) - \int_{\{(x \cdot n')=0\}}K(x)d\mathcal{H}^{d-1}(x) \right|\\
&\quad \leq 2 \int_{\{(x \cdot n)=0\}} |K(x) - K(Ox)| d\mathcal{H}^{d-1}(x).
\end{align*}
Since $|n-n'|<\delta'$ we have $|x - Ox| < \delta'|x|$. We then can apply the decaying Lipschitz bound in \textit{(II)} to get
\begin{align}
\left| \frac{1}{\mu_K(n)} - \frac{1}{\mu_K(n')} \right| &\leq 2 \int_{\{(x \cdot n)=0\}} C \min(\Lip(K), |x|^{-d-3-\delta}) \delta' |x| d\mathcal{H}^{d-1}(x) \notag\\
\label{eq:mobility continuity}&= \delta' 2C \int_{\R^{d-1}} \min(\Lip(K)|X|, |X|^{-d-2-\delta})dX.
\end{align}
Since this integral is finite we have proven that $1/\mu_K$ is uniformly continuous. Thus $1/\mu_K$ has a positive minimum. Due to the asymptotic and total bound on $K$, there has to be a domain $D$ on a large ball, with measure bounded from below, on any hyperplane $\{(x \cdot n)=0\}$, where $K$ is larger than a constant. $K$ being Lipschitz continuous, there are positive numbers $\tilde{r}$ and $\tilde{c}$ with 
\begin{align*}
K|_{D+[0,\tilde{r}]n} > \tilde{c}
\end{align*}
and for $h$ small enough $D+[0,\tilde{r}]n \subs B_{h^{-1/8+\eps'/2}}(0)$. Then
\begin{align*}
\int_{B_{h^{-1/8+\eps'/2}}(0) \cap \{0 \leq |y| < ch^{1/4+\eps'}\}} K(x) dx &> \int_{(D+[0,\tilde{r}]n) \cap \{ 0 < y < ch^{1/4+\eps'}\}} K(x) dx\\
&\geq C'''\, \tilde{c} h^{1/4+\eps'}.
\end{align*}
The statement is then true if
\begin{align*}
C'' h^{(1/8-\eps'/2)(2+\eps)} < C''' \, \tilde{c} h^{1/4+\eps'},
\end{align*}
which is equivalent to
\begin{align*}
h^{\eps/8-\eps'/2} < C'''',
\end{align*}
which, for all $\eps$ is true for $\eps'$ and $h$ uniformly small enough.
\end{proof}

\begin{proof}[Proof of Theorem~\ref{negative:Ker} in case of (I)]
In the proof of Theorem~\ref{anis:main} \textit{(I)} we used the non-negativity of $K$ only to show the a priori bound of $z$. It thus remains to prove that $z$ is in $\Oc(h)$. We proceed similarly to the non-negative case. Since $E_0$ is bounded and of class $C^2$, there is a positive number $R$ such that for every $x \in \partial E_0$
\begin{align*}
\{x\} = \overline{B_R(y)} \cap \partial E_0 \text{ for }p = x \pm R\nu(x).
\end{align*}
Without loss of generality, we fix some pair $(x,p)$, with $p \in E_0$, and a coordinate system, such that $\nu(x)=e_d$ and $x-2che_d=0$ for some $c$ with
\begin{align*}
c > \frac{L^2}{2R},
\end{align*}
where $L$ is such that $B_L(0)\sups \supp(K)$. Then $B_{\sqrt{h}L}(0) \sups \supp(K_h)$ and, for $h$ small enough, 
\begin{align*}
B_{L\sqrt{h}}(0)\cap \{y < 0\} \subs B_R(y) \subs E_0.
\end{align*}
For a visualization see Figure~\ref{fig:star shaped}. $K_h$ being an even function with support in $B_{\sqrt{h}L}(0)$ and $\int_{\R^d}K_h=1$, we have
\begin{align*}
\int_{B_{L\sqrt{h}}(0)\cap \{y < 0\}} K_h(x) dx =\frac{1}{2}.
\end{align*}
We remember that $0$ is in $E^{K,h}_1$ if and only if
\begin{align*}
(K_h \ast \chi_{E_0})(0) = \int_{E_0} K_h(x) dx > \frac{1}{2}.
\end{align*}
Thus, $0$ is in $E_1^{K,h}$ if and only if 
\begin{align*}
\int_{E_0 \cap \{y \geq 0\}} K_h(x) dx = \int_{E_0 \cap B_{\sqrt{h}L}(0) \cap \{y \geq 0\}} K_h(x) dx > 0.
\end{align*}
This is true for all $h$ small enough if the domain of integration $A \coloneqq E_0 \cap B_{\sqrt{h}L}(0) \cap \{y \geq 0\}$ is star-shaped with respect to $0$ for all $h$ small enough.\\

Let us suppose that the domain is star-shaped with respect to $0$, then we we can express the integral in polar coordinates. We define a function $S$: $\Sd \rightarrow \R_{\geq 0}$ such that 
\begin{align*}
E_0 \cap B_{\sqrt{h}L}(0) \cap \{y \geq 0\} = \{x \in \R^d | x = s\theta \text{, with } s<S(\theta)\}.
\end{align*}

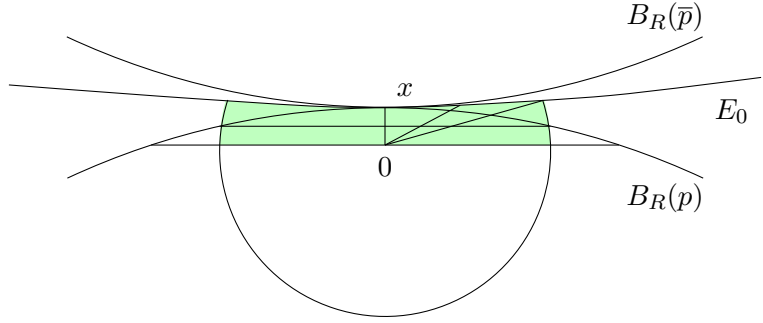
\begin{figure}
\centering
\begin{tikzpicture}

\fill [green!25] (0,-.5) circle (2.2);

\fill [white] plot [smooth] coordinates {(-5,0.3) (-3,.15) (-1,0.02) (0,0) (1,0.03) (3,0.16) (5,0.4)} -- (1,1.7) -- (-1,1.7) -- (-5,0.3);

\fill [white] (-2.2,-0.5) -- (-2.2,-2.7) -- (2.2,-2.7) -- (2.2,-0.5) -- (-2.2,-0.5);

\draw plot [smooth] coordinates {(-5,0.3) (-3,.15) (-1,0.02) (0,0) (1,0.03) (3,0.16) (5,0.4)};

\draw (4.23, -0.94) arc (65:115:10);

\draw (-4.23, 0.94) arc (245:295:10);

\draw (0,0) -- (0,-.5);

\draw (-3.11,-.5) -- (3.11,-.5);

\draw (-2.2,-0.25) -- (2.2,-0.25);

\draw (-2.09,0.1) arc (162:378:2.2);

\draw (0,-.5) -- (1,0.03);

\draw (0,-.5) -- (2.09,0.09);

\node [above right=1pt of {(0,0)}] {$x$};

\node [below=1pt of {(0,-.5)}] {$0$};

\node [below right =2pt of {(4.2,.3)}] {$E_0$};

\node [below right=3pt of {(3,-.8)}] {$B_R(p)$};

\node [above right=3pt of {(3,.8)}] {$B_R(\overline{p})$};
\end{tikzpicture}
\caption{The support of $K_h$ between the hyperplane and $\partial E_0$ is star shaped with respect to $0$}
\label{fig:star shaped}
\end{figure}
By construction this set contains $B_{\sqrt{h}L}(0) \cap \{0 \leq y < ch\}$. Since for all $\theta \in \mathbb{S}^{d-1}$ and $r>0$
\begin{align*}
\int_0^r s^{d-1} K(s\theta) ds \geq 0
\end{align*}
we obtain that for any $D \subs \Sd$
\begin{equation}
\label{eq:star-shaped complement}
\int_D \int_0^{S(\theta)} s^{d-1} K_h(s\theta) ds\, d\mathcal{H}^{d-1}(\theta) \geq 0.
\end{equation}
We know there to be a $\tilde{\theta} \in \Sd$ such that for all $r>0$
\begin{align*}
\int_0^r s^{d-1} K(s\tilde{\theta}) ds > 0.
\end{align*}
The kernel being even we can choose $\tilde{\theta}$ such that $(\tilde{\theta} \cdot e_d) \geq 0$, where we know the function $S$ to be positive. $K$ and thus $K_h$ being H\"older continuous implies, that there is a neighborhood of $\tilde{\theta}$, which we call $\tilde{D}$ with
\begin{equation}
\label{eq:star-shaped positive}
\int_{\tilde{D}} \int_0^{S(\theta)} s^{d-1} K_h(s\theta) ds \, d\mathcal{H}^{d-1}(\theta) > 0.
\end{equation} 
Using \eqref{eq:star-shaped positive} and \eqref{eq:star-shaped complement} with $D= (\tilde{D})^c$ we obtain
\begin{align*}
\int_A K_h(x)dx = \int_{\Sd} \int_0^{S(\theta)} s^{d-1} K_h(s\theta) ds \, d\mathcal{H}^{d-1}(\theta) >0.
\end{align*}

It remains to show that $A$ is star-shaped with respect to $0$ for $h$ small enough. We remember the surface function $g$ and observe that $A$ is star-shaped with respect to $0$ if for all $X \in B_{\sqrt{h}L}^{d-1}$ it holds that
\begin{align*}
2ch + g(X) \geq \partial_X g(X).
\end{align*}
On the finite ball $B_{\sqrt{h}L}^{d-1}$ we can estimate $|g|$ and $|\partial_X g(X)|$ by $C|X|^2$ using Taylor's theorem as before. Thus we can estimate
\begin{align*}
(\partial_X g(X) - g(X))|_{B_{\sqrt{h}L}^{d-1}(0)} \leq (2C|X|^2)|_{B_{\sqrt{h}L}^{d-1}(0)} \leq 2CL^2h.
\end{align*}
Since we can always choose $c$ to be larger in a uniform way and $C$ and $L$ are uniformly bounded properties of $E_0$ and $K$ respectively we choose
\begin{align*}
c = L^2 \max \left( \frac{1}{2R}, C \right).
\end{align*}
This concludes the proof.
\end{proof}

\begin{proof}[Proof of Theorem~\ref{negative:Ker} in case of (II)]
We again only have to show the a priori bound; here this means that $z$ is in $\Oc(h^{3/4+\eps'})$ for some $\eps' >0$.  With $E_0$ being a bounded set of class $C^2$, there is a number $R>0$, such that for every $x \in \partial E_0$ and
\begin{align*}
p = x \pm R\nu(x)
\end{align*}
we have
\begin{align*}
\{x\} = \overline{B_R(p)} \cap \partial E_0.
\end{align*}
Without loss of generality, we fix some pair $(x,p)$, with $p \in E_0$, a number $\eps'>0$ and a coordinate system, such that $\nu(x) = e_d$ and $x-2ch^{3/4+\eps'}e_d = 0$ for some $c$ with $c>1/(2R)$. Thus there is some $h_0>0$ such that for all $h \in(0,h_0)$
\begin{align*}
B_{h^{3/8+\eps'/2}}(0) \cap \{y<ch^{3/4+\eps'}\} \subs B_R(p) \subs E_0.
\end{align*}
Analogously, we show that
\begin{equation}
\label{eq:surface upper bound}
B_{h^{3/8+\eps'/2}}(0) \cap \{y \geq 3ch^{3/4+\eps'}\} \subs E_0^c.
\end{equation} 

Now we want to show that $\int_{E_0}K_h(x) dx > 1/2$. Since $\int_{\{y<0\}}K_h(x)dx = 1/2$, we compare $E_0$ to the lower half space. It is then enough to show that $\int_{E_0 \cap \{y>0\} \cap B_{3/8 + \eps'/2}(0)} K_h(x) dx$ is larger than the absolute mass of $K_h$ outside of the ball $B_{3/8 + \eps'/2}(0)$, where we do not generally know whether a given point is in $E_0$. We thus want to show that for $h$ small enough
\begin{equation}
\label{eq:main comparison negative kernel}
\int_{E_0 \cap \{y>0\} \cap B_{3/8 + \eps'/2}(0)} K_h(x) dx > \int_{(B_{3/4 + \eps'/2}(0))^c} |K_h(x)|dx.
\end{equation}
We first find a lower bound for the left-hand side of \eqref{eq:main comparison negative kernel}. To this end, we again first claim the set $A \coloneqq E_0 \cap \{y>0\} \cap B_{3/8 + \eps'/2}(0)$ to be star-shaped with respect to $0$. If this is true, we can again construct a function $S$ with
\begin{align*}
A = \left\{ x \in \R^d \Bigg| |x| < S\left(\frac{x}{|x|}\right) \right\}.
\end{align*}
Then we know for all measurable sets $D \subs \Sd$ that
\begin{align*}
\int_D \int_0^{S(\theta)} s^{d-1} K_h(s \theta)\, ds \, d\mathcal{H}^{d-1}(\theta) \geq 0.
\end{align*}
Since $K$ is continuous and $K(0)>0$ there are positive real numbers $\tilde{r}$ and $\tilde{c}$ with $K|_{B_{\tilde{r}}(0)} \geq \tilde{c}$. Independent of direction, we then find a cylinder for $c',r'>0$ with
\begin{align*}
U \coloneqq \left\{ (X,y) \in \R^d \Big| |X|<r', 0<y<c' \right\} \subs B_{\tilde{r}}(0).
\end{align*}
For the contracted kernel $K_h$ this implies
\begin{align*}
K_h|_{\sqrt{h}\, U} \geq \tilde{c} h^{-d/2}.
\end{align*}
We construct a cone inside of this cylinder (see Figure~\ref{fig:positive cone})
\begin{align*}
V \coloneqq \conv\left( \{0\} \cup B_{\sqrt{h}r'}^{d-1}(0,3ch^{3/4 + \eps'}) \right) \subs \sqrt{h}\, U
\end{align*}
with $0 \in E_0$ and by \eqref{eq:surface upper bound} $B_{\sqrt{h}r'}^{d-1}(0,3ch^{3/4 + \eps'})$ is disjoint from $E_0$. This allows us to estimate the mass along the horizontal hyperplane in the upper half-space. We first use the assumed non-negativity of the integrals along rays to reduce the set to the cone:
\begin{align*}
\int_{A} K_h(x) dx &= \int_{\Sd} \int_0^{S(\theta)} s^{d-1} K_h(s\theta) \, ds \, d\mathcal{H}^{d-1}(\theta) \\
&\geq \int_{\{ \theta \in \Sd | \exists \lambda > 0: \lambda \theta \in V\}} \int_0^{S(\theta)} s^{d-1} K_h(s\theta) \, ds \, d\mathcal{H}^{d-1}(\theta)\\
&=\int_{E_0 \cap V} K_h(x) dx.
\end{align*}
On $E_0 \cap V$, the kernel $K_h$ is bounded from below by $\tilde{c}h^{-d/2}$. Furthermore, by construction
\begin{align*}
\{(X,y) \in V | y<ch^{3/4+\eps'}\} \subs E_0 \cap V.
\end{align*}
We calculate the volume of this cone, which has height $ch^{3/4+\eps'}$ and radius $(r'/3)h^{1/2}$. Thus there is a positive constant $C'$ with
\begin{align*}
\mathcal{L}^d\left( \{(X,y) \in V | y<ch^{3/4+\eps'}\} \right) \geq C'h^{d/2+1/4+\eps'}.
\end{align*}
We conclude for the lower bound that
\begin{align*}
\int_{E_0 \cap V} K_h(x) dx \geq \tilde{c} \, C' h^{1/4 + \eps'}.
\end{align*}
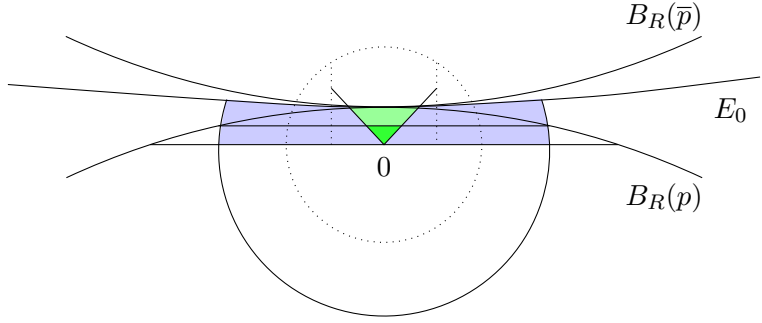
\begin{figure}
\centering
\begin{tikzpicture}

\fill[blue!20] (-2.1,0.1) -- (-1,0.02) -- (0,0) -- (1,0.03) -- (2.1,0.11) -- (2.18,-0.25) -- (2.2,-0.5) -- (-2.2,-0.5) --cycle;

\fill[green!80] (0,-.5) -- (.232,-.25) -- (-.232,-.25) -- cycle;

\fill[green!40] (.233,-.25) -- (-.233,-.25) -- (-.466,0) -- (.466,0) -- cycle;

\draw plot [smooth] coordinates {(-5,0.3) (-3,.15) (-1,0.02) (0,0) (1,0.03) (3,0.16) (5,0.4)};

\draw (4.23, -0.94) arc (65:115:10);

\draw (-4.23, 0.94) arc (245:295:10);


\draw (-3.11,-.5) -- (3.11,-.5);

\draw (-2.2,-0.25) -- (2.2,-0.25);

\draw (-2.09,0.1) arc (162:378:2.2);

\draw[dotted] (0,-.5) circle (1.3);

\draw[dotted] (-.7,-.5) -- (-.7,.59);

\draw[dotted] (.7,.59) -- (.7,-.5);


\draw (-.7,.25) -- (0,-.5) -- (.7,.25);


\node [below=1pt of {(0,-.5)}] {$0$};

\node [below right =2pt of {(4.2,.3)}] {$E_0$};

\node [below right=3pt of {(3,-.8)}] {$B_R(p)$};

\node [above right=3pt of {(3,.8)}] {$B_R(\overline{p})$};
\end{tikzpicture}
\caption{The positive mass we can estimate on the upper half-space. The mass of the blue area is non-negative due to the non-negativity of integrals over rays. The green cone has base diameter $\Oc(\sqrt{h})$ and on it $K$ is positively bounded from below. Using the minimal distance to $\partial E_0$ we estimate its volume.}
\label{fig:positive cone}
\end{figure}
Next we estimate the right-hand side of \eqref{eq:main comparison negative kernel} from above. By the general bound on $|K|$ we obtain that
\begin{align*}
\int_{(B_{h^{3/8+\eps'/2}}(0))^c} |K_h(x)| dx &= \int_{(B_{h^{-1/8+\eps'/2}}(0))^c} |K(x)| dx\\
&\leq \int_{(B_{h^{-1/8+\eps'/2}}(0))^c} |x|^{-d-2-\eps} dx\\
&= c' \int_{h^{-1/8+\eps'/2}}^{\infty} r^{-3-\eps} dr\\
&= c'' h^{(1/8-\eps'/2)(2+\eps)}.
\end{align*}
We conclude that $0$ is in $E_1^{K,h}$ if
\begin{align*}
c'' h^{(1/8-\eps'/2)(2+\eps)} < \tilde{c}C' h^{1/4+\eps'}.
\end{align*}
This is equivalent to 
\begin{align*}
h^{\eps/8-(1/2)(4+\eps)\eps'} < \frac{\tilde{c}C'}{c''},
\end{align*}
which is true for $\eps'$ and $h$ small enough only depending on $\eps$. It remains to show that the set $A \coloneqq (E_0 \cap \{y \geq 0\} \cap B_{h^{3/8+\eps'/2}}(0))$ is star-shaped with respect to $0$. This is again true if for all $X \in B_{h^{3/8+\eps'/2}}^{d-1}(0)$ it holds that
\begin{align*}
ch^{3/4+\eps'} + g(X) \geq \partial_X g(X).
\end{align*}
Still on a finite ball, we can again obtain $C|X|^2$ as a general bound of $|g|$ and $|\partial_X g(X)|$. Thus we have that
\begin{align*}
(|g(X)|+|\partial_X g(X)|)|_{ B_{h^{3/8+\eps'/2}}^{d-1}(0)} \leq 2Ch^{3/4+\eps'}.
\end{align*}
As $C$ depends only on $g$, this yields the claim if we choose $c=2C$.
\end{proof}

\begin{proof}[Proof of Proposition~\ref{prop:monotonicity of energies} on a bound for the adjusted $K_h$-perimeter]
Sets of finite perimeter can be approximated in measure and anisotropic perimeter by sets that are smooth and bounded \cite[Theorems 13.8 and 20.6]{Maggi}. We thus assume $E$ to be such a set. From the proof of Proposition~\ref{anis:per} we know that
\begin{align*}
P_{K,h}(E) &= \int_{\Rd} K(y) \frac{1}{\sqrt{h}} \int_{\Rd} \chi_E(x-\sqrt{h}y) \chi_{E^c}(x) dx\, dy \\
&= \int_{\Sd} \int_0^{\infty} r^d K(r\theta) dr \frac{1}{\sqrt{h}} \int_{E^c} \chi_E(x - \sqrt{h}\theta) dx\, d\mathcal{H}^{d-1}(\theta) 
\end{align*}
and
\begin{align*}
P_{\sigma_K}(E) &= \int_{\Rd} K(y) \frac{1}{2} \int_{\partial^* E} |\nu(x) \cdot y| d\mathcal{H}^{d-1}(x)\, dy \\
&= \int_{\Sd} \int_0^{\infty} r^d K(r\theta) dr \int_{\partial^* E}(\nu(x) \cdot \theta)_+ d\mathcal{H}^{d-1}(x)\, d\mathcal{H}^{d-1}(\theta).
\end{align*}
Since for all $\theta \in \Sd$
\begin{align*}
\int_0^{\infty} r^d K(r\theta) dr \geq 0
\end{align*}
it is enough to show
\begin{align*}
\frac{1}{\sqrt{h}} \int_{E^c} \chi_E(x-\sqrt{h}\theta) dx \leq \int_{\partial^* E} (\nu(x) \cdot \theta)_+ d\mathcal{H}^{d-1}(x).
\end{align*}
For fixed $\theta$, we first rewrite the integral as the measure of a set
\begin{align*}
\int_{\Rd} \chi_E(x-\sqrt{h}\theta) \chi_{E^c}(x) dx = \mathcal{L}^d\left(\{x \in E^c | x-\sqrt{h}\theta \in E\}\right).
\end{align*}
The line from $x-\sqrt{h}\theta$ to $x$ crosses the boundary of $E$ outwardly. Since $E$ is smooth we define $(\partial E)_{\theta}^+ \coloneqq \{y \in \partial E | (\nu_E(y) \cdot \theta) >0\}$ and
\begin{align*}
&\mathcal{L}^d\left(\{x \in E^c | x-\sqrt{h}\theta \in E\}\right)\\
&\leq  \mathcal{L}^d\left(\{x \in \R^d | \exists y \in (\partial E)_{\theta}^+, t \in [0,1]: x = y + \sqrt{h} t \theta \}\right).
\end{align*}
Since $E$ has smooth boundary, $(\partial E)_{\theta}^+$ is $\mathcal{H}^{d-1}$-measurable and can be split into countably many subsets $A_l$, such that each can be parametrized as the graph of a function $f_l$ on $G_l$ in a coordinate system $\tau_i$, where $\theta = \tau_d$. We again write $X=(x \cdot \tau_1, \ldots, x \cdot \tau_{d-1})$. Thus
\begin{align*}
& \mathcal{L}^d\left(\{x \in \R^d | \exists y \in (\partial E)_{\theta}^+, t \in [0,1]: x = y + \sqrt{h} t \theta \}\right)\\
&\leq \sum_{l \in \mathbb{N}}  \mathcal{L}^d\left(\{x \in \R^d | \exists y \in A_l, t \in [0,1]: x = y + \sqrt{h} t \theta \}\right) \\
&= \sum_{l \in \mathbb{N}} \int_{G_l} (f_l(X) + \sqrt{h} - f_l(X)) dX \\
&= \sqrt{h} \sum_{l \in \mathbb{N}}  \mathcal{L}^{d-1}\left( G_l \right).
\end{align*}
Since $f_l$ are smooth functions they are Lipschitz continuous on compact subsets of $G_l$. Additionally, $(\nu_E(x) \cdot \theta)$ is positive and measurable and we can represent the outer normal in terms of $f$ as
\begin{align*}
\nu_E(x) = \frac{(\partial_{\tau_1} f(x), \ldots, \partial_{\tau_{d-1}} f(x), 1)}{\sqrt{1 + |\nabla f(x)|^2}}.
\end{align*}
We obtain
\begin{align*}
(\nu_E(x) \cdot \theta)|_{A_l} = (\nu_E(f(X)) \cdot \tau_d) = \left( 1 + |\nabla f(X)|^2 \right)^{-1/2} = (Jf(X))^{-1}.
\end{align*}
Thus we can apply the area function for injective maps on a disjoint compact covering of $G_l$ and
\begin{align*}
\int_{A_l} (\nu_E(x) \cdot \theta) d\mathcal{H}^{d-1}(x) &= \int_{f_l(G_l)} (\nu_E(x) \cdot \theta) d\mathcal{H}^{d-1}(x) \\
&=  \int_{G_l} (\nu_E(f(X)) \cdot \tau_d) Jf(X) dX \\
&= \int_{G_l} (Jf(X))^{-1} Jf(X) dX\\
&=  \mathcal{L}^{d-1}\left( G_l \right).
\end{align*}
Collecting these results we obtain that since $(A_l)_{l\in \mathbb{N}}$ is a disjoint covering of $(\partial E)_{\theta}^+$
\begin{align*}
\frac{1}{\sqrt{h}} \int_{E^c} \chi_E(x-\sqrt{h}\theta) dx &\leq \sum_{l \in \mathbb{N}}  \mathcal{L}^{d-1}\left( G_l \right) \\
&= \sum_{l \in \mathbb{N}} \int_{A_l} (\nu_E(x) \cdot \theta) d\mathcal{H}^{d-1}(x) \\
&= \int_{(\partial E)_{\theta}^+} (\nu_E(x)) \cdot \theta) d\mathcal{H}^{d-1}(x) \\
&= \int_{\partial E}  (\nu_E(x)) \cdot \theta)_+ d\mathcal{H}^{d-1}(x).
\end{align*}
This concludes the proof for bounded smooth sets. By \cite[Theorem 13.8]{Maggi}, for any set of finite perimeter $E$ with $|E|<\infty$ there is a sequence of bounded sets with smooth boundary $E_s$ converging to $E$ in measure and perimeter. By \cite[Theorem 20.6]{Maggi}, since $\sigma_K$ is continuous, $P_{\sigma_K}(E_s)$ converges to $P_{\sigma_K}(E)$. It remains to show that for fixed $h$ and $K$, $E_s$ converges in $P_{K,h}$. We estimate
\begin{align*}
&|P_{K,h}(E) - P_{K,h}(E_s)|\\
&\leq \frac{1}{\sqrt{h}} \int_{\Rd} \int_{\Rd} K_h(x-y) \left| \chi_E(x)\chi_{E^c}(y) - \chi_{E_s}(x)\chi_{E_s^c}(y) \right| dx dy,
\end{align*}
which is bounded by
\begin{align*}
\frac{1}{\sqrt{h}} \int_{\Rd} \int_{\Rd} |K_h(x-y)| \left( \chi_{E \triangle E_s}(x) \chi_{E^c \cup E_s^c}(y) + \chi_{E \cup E_s}(x) \chi_{E^c \triangle E_s^c}(y) \right) dx dy.
\end{align*}
Since $E \triangle E_s = E^c \triangle E_s^c$ and due to a reciprocity of impact argument we can estimate this by
\begin{align*}
\frac{2}{\sqrt{h}} \int_{\Rd} \int_{\Rd} |K_h(x-y)| \chi_{E \triangle E_s}(x)dy dx = \frac{2\|K\|_{L^1}}{\sqrt{h}} |E \triangle E_s|,
\end{align*}
which converges to zero. We obtain 
\begin{align*}
P_{K,h}(E) = \underset{s \rightarrow 0}{\lim}P_{K,h}(E_s) \leq \underset{s \rightarrow 0}{\lim}P_{\sigma_K}(E_s) = P_{\sigma_K}(E),
\end{align*}
which concludes the proof.
\end{proof}

\section{Kernel constructions}
\label{sec:kernel construction}

\subsection{Tension- and mobility generating functions}
\label{subsec:tension and mobility}

In Section~\ref{sec:initial step}, we analyzed the connection between kernel and approximated movement. Given a kernel, we can calculate the resulting motion law directly. But more often, we have a given motion law and want to construct a kernel approximating this movement. Thus we have given a surface tension $\sigma$ and mobility function $\mu$ and want to construct a convolution kernel $K$. This can be expressed as two inverse problems.\\

We have already seen that $K$ depends on $\sigma$ and $\mu$ only through the tension generating directional distribution
\begin{align*}
A(\theta) = \int_0^{\infty} r^d K(r\theta) dr
\end{align*}
and the mobility generating directional distribution 
\begin{align*}
B(\theta) = 2\int_0^{\infty} r^{d-2} K(r\theta) dr.
\end{align*}
If we have obtained $A$ and $B$ we can explicitly calculate a kernel $K$. Note that there are infinitely many options. To be able to influence $A$ and $B$ separately we need to be able to shift mass inwards and outwards. We choose for some functions $f$,$g$: $\mathbb{S}^{d-1} \rightarrow \R_{\geq 0}$, with $r \in \R_{\geq 0}$ and $\theta \in \Sd$
\begin{align*}
K(r\theta) \coloneqq f(\theta) \left[ r^2(g(\theta) - r)^2 \right] \chi_{[0,g(\theta)]}(r).
\end{align*}
In any direction, this is a polynomial cut off in its roots. To find $f$ and $g$, we first calculate
\begin{align*}
\int_0^{a}r^m(a-r)^2 dr = \frac{2a^{m+3}}{m^3 + 6m^2 +11m+6}.
\end{align*}
We can then calculate $A$ and $B$ for this kernel:
\begin{align*}
A(\theta) &= \frac{2f(\theta)g(\theta)^{d+5}}{d^3 + 12d^2 + 47d + 60}, \\
B(\theta) &= \frac{4f(\theta)g(\theta)^{d+3}}{d^3 + 6d^2 + 11d + 6}.
\end{align*}
By division of functions we obtain 
\begin{align*}
g(\theta) = \sqrt{2\frac{d^3 + 12d^2 + 47d + 60}{d^3 + 6d^2 + 11d + 6} \frac{A(\theta)}{B(\theta)}}.
\end{align*}
We can then input our result for $g$ in the equation for $B$ to obtain
\begin{align*}
f(\theta) = 2^{-(d+7)/2} \frac{\left( (d^3 + 6d^2 + 11d + 6)B(\theta) \right)^{(d+5)/2}}{\left( (d^3 + 12d^2 + 47d + 60) A(\theta) \right)^{(d+3)/2}}.
\end{align*}
We obtain that kernels constructed in this way are applicable to the results of this paper.
\begin{prop}
Let $A$ and $B$: $\Sd \rightarrow \R_+$ be positive, even and H\"older continuous functions. Then a kernel $K$, constructed in the given manner is non-negative, even, compactly supported and H\"older continuous, thus the normalized kernel $K/\|K\|_1$ satisfies the conditions of Theorem~\ref{anis:main}.
\end{prop}
\begin{proof}
We observe that $A$ and $B$ map the compact set $\Sd$ continuously on $\R_+$. Thus $\Im(A)$ and $\Im(B)$ are compact subsets of the positive numbers. Restricted to these sets, the dependence of $f$ and $g$ on $A$ and $B$ is Lipschitz continuous. Moreover, the images of $f$ and $g$ are again compact subsets of the positive numbers and both functions are even. Then $K$ is non-negative, even and compactly supported. Since the cut-off occurs in the roots of the function, the dependence of $K$ on $f$ and $g$, on the compact subset of the positive numbers, is Lipschitz continuous. Thus the dependence of $K$ on $A$ and $B$ is Lipschitz continuous. The H\"older continuity of $A$ and $B$ implies that $K$ is H\"older continuous.
\end{proof}

We still want to obtain $A$ and $B$ from $\sigma$ and $\mu$ respectively. In two dimensions the connection between mobility and mobility generating directional distribution 
\begin{align*}
\frac{1}{\mu(\theta)} = \int_{ \{x \in \mathbb{S}^{d-1} | x \cdot \theta = 0 \} } B(x) d\mathcal{H}^{d-2}
\end{align*}
becomes trivial as, for $\theta^{\perp}$ one of the two vectors orthogonal to $\theta$, both having equal value on the even function $\mu$,
\begin{align*}
\frac{1}{\mu(\theta^{\perp})} = 2B(\theta).
\end{align*}
The function $g$ then simplifies to
\begin{align*}
g(\theta) = \sqrt{14 A(\theta) \mu(\theta^{\perp})}
\end{align*}
and subsequently, $f$ simplifies to
\begin{align*}
f(\theta) = \frac{15}{2} (14 A(\theta))^{-5/2} \left( \frac{1}{\mu(\theta^{\perp})} \right)^{7/2}.
\end{align*}
It remains to analyze how to obtain $A$ from
\begin{align*}
\sigma(\nu) = \frac{1}{2} \int_{\mathbb{S}^{d-1}} |\nu \cdot \theta| A(\theta) d\mathcal{H}^{d-1}(\theta).
\end{align*}
In two dimensions we can parametrize the circle, and since $A$ is even
\begin{align*}
\sigma(x) &= \frac{1}{2} \int_0^{2\pi} |\cos(y-x)| A(y) dy\\
&= \int_0^{2\pi} (\cos(y-x))_+ A(y) dy = \int_{-\pi/2}^{\pi/2} \cos(y) A(y+x) dy.
\end{align*}
Noticing the linear dependency between $\sigma$ and $A$ and that $\tilde{A} = A(\cdot + a)$ generates $\tilde{\sigma} = \sigma(\, \cdot \, + a)$, we can choose a class of functions containing a basis and calculate the tension functions generated by functions of this class. Approximating $\sigma$ with the calculated basis then approximates the corresponding tension generating directional distribution $A$. A classical choice for such a class of measurable even functions on the $1$-sphere, parametrized on $[0,2\pi)$, is
\begin{align*}
A_{a,b}(\cdot) \coloneqq \chi_{(-b,b)\cup(\pi-b,\pi+b)}(\, \cdot \, + a)
\end{align*}
for all $a \in [0,\pi)$ and $b \in [0,\pi/2]$. As $\tilde{A} = A(\cdot + a)$ generates $\tilde{\sigma} = \sigma(\, \cdot \, + a)$, let $a$ be zero w.l.o.g. We calculate 
\begin{align*}
\sigma_{A_{0,b}}(x) &= \frac{1}{2} \int_{(-b,b)\cup(\pi-b,\pi+b)} |\cos(y-x)| dy\\
&= 2\sgn(\cos(x + b)) \sin(x) \cos(b) + 2\chi_{ \{\cos(x)^2 > \sin(b)^2\}}(x).
\end{align*}

\begin{figure}
\centering
\includegraphics[width=0.5\textwidth]{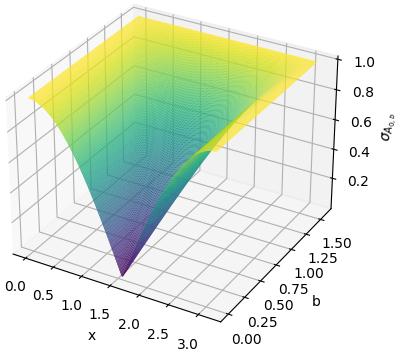}
\caption{Surface tensions generated by $A_{0,b}$}
\label{fig:surf tension 3d plot}
\end{figure}

We plot this one parameter family of functions, normalized in the $C^0$-norm in Figure~\ref{fig:surf tension 3d plot}. Any surface tension $\sigma$ that is generated by a non-negative $L^1$-function $A$ can be constructed as a sum of weighted and shifted profiles
\begin{align*}
\sigma(x) = \sum_{k \in \mathbb{N}} c_k \, \sigma_{A_{0,b_k}}(x-a_k)
\end{align*}
for sequences $(a_k) \in [0,\pi)^{\infty}$, $(b_k) \in (0,\pi/2]^{\infty}$ and $(c_k) \in (\R_{\geq 0})^{\infty}$.\\

\subsection{A kernel inducing backwards-in-time motion by mean curvature}
\label{subsec:backwards-in-time}

Next, we construct kernels of negative initial motion. By Theorem~\ref{negative:Ker} some partially negative kernels have uniformly converging initial motion. In particular, this is true for $C_c^{0,\alpha}$-kernels with
\begin{align*}
\int_0^r s^{d-1} K(s\theta) ds > 0
\end{align*}
for all $\theta \in \Sd$ and $r>0$. But the surface tension depends on the first moment of the kernel. Thus we can construct kernels with negative mass on the outside to generate a kernel satisfying all conditions that is associated to a negative surface tension $\sigma_K$. Let for example $K$ be rotation invariant and in any direction be the sum of three hat functions, cf.\ Figure~\ref{fig:3hats}, weighed such that
\begin{align*}
\int_{\R^d} K(x) dx &= 1,\\
\frac{1}{\mu_K(n)} = 2\int_{(n \cdot x)=0} K(x) d\mathcal{H}^{d-1}(x) &=1,\\
\sigma_K(n) = \frac{1}{2}\int_{\R^d} |n \cdot x| K(x) dx &=-1.
\end{align*}
Since $K$ is rotation invariant we write $\tilde{K}(|x|)\coloneqq K(x)$. The conditions then can be expressed as
\begin{align*}
\int_0^{\infty} r^{d-2} \tilde{K}(r) dr &= \frac{\Gamma((d-1)/2)}{4 \pi^{(d-1)/2}},\\
\int_0^{\infty} r^{d-1} \tilde{K}(r) dr &= \frac{\Gamma(d/2)}{2 \pi^{d/2}},\\
\int_0^{\infty} r^d \tilde{K}(r) dr &= -\frac{2 (d-1) \Gamma((d-1)/2)}{\pi^{(d-1)/2}}.
\end{align*}
We then can test positions of the three hat functions such that the first weight is positive and the third is negative. Since we require a $\tilde{\theta} \in \Sd$ with
\begin{align*}
\int_0^r s^{d-1} K(s\tilde{\theta}) ds > 0
\end{align*}
for all $s>0$, the first position has to be zero. In two dimensions a solution is
\begin{align*}
\tilde{K}(r) = &\frac{185\pi-384}{576\pi}(1-|2r-1|)^+ + \frac{139\pi + 240}{144\pi}(1-|2r-7|)^+\\
&- \frac{151\pi + 192}{192\pi}(1-|2r-9|)^+.
\end{align*}
While the initial motion of this kernel converges to that of backwards-in-time mean curvature flow, we suspect the process to be unstable. A plane would remain stationary and a ball would expand uniformly. But a surface oscillating on a large scale would gain sharp spikes and surfaces repel each other.\\

\begin{figure}
\centering
\begin{tikzpicture}

\draw [-latex] (0,-1.5) -- (0,2);

\draw [-latex] (-0.1,0) -- (5.5,0);

\draw [red] (0,0) -- (0.5,0.109) -- (1,0) -- (3,0) -- (3.5,1.5) -- (4,0) -- (4.5,-1.105) -- (5,0) -- (5.2,0);

\draw (-0.1,-1) -- (0,-1);

\draw (-0.1,1) -- (0,1);

\node [left = 0pt of {(0,1)}] {$1$};

\node [left = 0pt of {(0,0)}] {$0$};

\node [left = 0pt of {(0,-1)}] {$-1$};

\node [below = 0pt of {(1,0)}] {$1$};

\node [below = 0pt of {(2,0)}] {$2$};

\node [below = 0pt of {(3,0)}] {$3$};

\node [below = 0pt of {(4,0)}] {$4$};

\node [below = 0pt of {(5,0)}] {$5$};

\end{tikzpicture}
\caption{The function $\tilde{K}$ - the relief of the kernel}
\label{fig:3hats}
\end{figure}
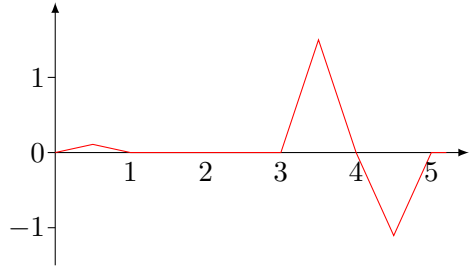

\subsection{Fine properties of thresholding and fattening of level sets}
\label{subsec:fine properties}

We remember Esedo\u{g}lu and Otto's minimization problem \eqref{variational}
\begin{align*}
&\underset{E \subs \Rd}{\argmin}\left\{ \frac{1}{\sqrt{h}} \int_{E^c} K_h \ast \chi_E dx + \frac{1}{\sqrt{h}} \int_{\Rd} (\chi_E - \chi_{E_k}) K_h \ast (\chi_E - \chi_{E_k}) dx \right\}\\
=\, &\underset{E \subs \Rd}{\argmin}\left\{ \int_E \left( 1 - 2(K_h \ast \chi_{E_k}) \right) dx \right\}.
\end{align*}
Up to null sets, any minimizer contains the set resulting from thresholding $T_{K_h}E_k = \{x \in \Rd | (K_h \ast \chi_{E_k})(x) > 1/2 \}$. But the minimization is indifferent on the set $(T_{K_h}E_k)^c \setm T_{K_h}(E_k^c) = \{x \in \Rd | (K_h \ast \chi_{E_k})(x) = 1/2\}$. For the classical thresholding algorithm this level set is a nullset. Indeed, since the Gaussian is an analytical function and for a bounded set $E_k$, the function $\chi_{E_k}$ is measurable and bounded with compact support, the convolution is again an integrable analytical function. The level set is thus a nullset and has no impact on the process. \\

On more general classes of kernels we find examples, where the level set has positive measure, see Examples~\ref{ex: vanishing}, \ref{ex: pointsymmetric}, \ref{ex: decreasing} below. This means that the thresholding algorithm \eqref{eq:def_thresholding} is, for general kernels, not symmetric with respect to the set complement. Esedo\u{g}lu and Otto's minimization problem \eqref{variational}, however, is symmetric with respect to the set complement. This discrepency can be explained easily: As indicated above, the minimization problem possibly has a large class of solutions. Thresholding chooses the smallest solution with respect to set inclusion (up to null sets). We firstly demonstrate the possible size of the level set.
\begin{example}
\label{ex: vanishing}
Let $E = \{ x \in  \Rd | \left \lfloor{x_d}\right \rfloor/2 \in \mathbb{Z} \}$ and $K = 2^{-d} \chi_{[-1,1]^d}$, see Figure~\ref{fig:constconv}. Then $(K \ast \chi_{E}) \equiv 1/2$ and $T_KE = T_K(E^c) = \varnothing$. Moreover, we can choose $\tilde{E} \coloneqq E \cap U$, for $U$ a large bounded domain. Then $(K\ast \chi_{\tilde{E}})(x) = 1/2$ if $\dist(x,U^c)\geq 2^{1/d}$.
\end{example}
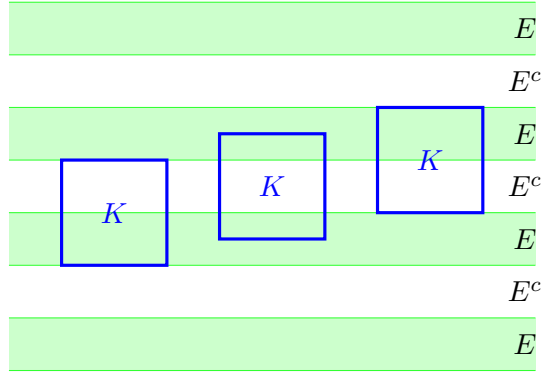
\begin{figure}
\centering
\begin{tikzpicture}[scale=.7]
\fill[green!20] (-5,0) -- (5,0) -- (5,1) -- (-5,1) -- cycle;

\fill[green!20] (-5,2) -- (5,2) -- (5,3) -- (-5,3) -- cycle;

\fill[green!20] (-5,4) -- (5,4) -- (5,5) -- (-5,5) -- cycle;

\fill[green!20] (-5,6) -- (5,6) -- (5,7) -- (-5,7) -- cycle;

\draw[green!80] (-5,0) -- (5,0);

\draw[green!80] (-5,1) -- (5,1);

\draw[green!80] (-5,2) -- (5,2);

\draw[green!80] (-5,3) -- (5,3);

\draw[green!80] (-5,4) -- (5,4);

\draw[green!80] (-5,5) -- (5,5);

\draw[green!80] (-5,6) -- (5,6);

\draw[green!80] (-5,7) -- (5,7);

\draw[blue, very thick] (-4,2) -- (-2,2) -- (-2,4) -- (-4,4) -- cycle;

\draw[blue, very thick] (-1,2.5) -- (1,2.5) -- (1,4.5) -- (-1,4.5) -- cycle;

\draw[blue, very thick] (2,3) -- (4,3) -- (4,5) -- (2,5) -- cycle;

\node[blue] at (-3,3) (a) {$K$};

\node[blue] at (0,3.5) (b) {$K$};

\node[blue] at (3,4) (c) {$K$};

\node at (4.8,.5) {$E$};

\node at (4.8,1.5) {$E^c$};

\node at (4.8,2.5) {$E$};

\node at (4.8,3.5) {$E^c$};

\node at (4.8,4.5) {$E$};

\node at (4.8,5.5) {$E^c$};

\node at (4.8,6.5) {$E$};
\end{tikzpicture}
\caption{In any position the convolution evaluates $1/2$}
\label{fig:constconv}
\end{figure}
One may suspect this to result from the kernel $K$ not being smooth. However, our next example is a smooth and point-symmetric kernel. To construct it, we first define the one- and $d$-dimensional bump functions
\begin{align*}
\eta(s) \coloneqq \exp \left( \frac{s^2}{s^2 - 1} \right) \chi_{[-1,1]}(s), \quad \eta_d(x) \coloneqq \prod_{i=1}^d \eta(x_i).
\end{align*}
We use this to additionally construct a function that smoothly connects the constant functions $0$ and $1$
\begin{align*} 
\tilde{\eta}(s) \coloneqq \eta(\eta(s-1))\chi_{[0,1]}(s)
\end{align*}
\begin{example}
\label{ex: pointsymmetric}
Let $E=B_{2^{-1/d}}(0)$ and for some small positive $h$ we set
\begin{align*}
\tilde{K}(x)=(\chi_{B_1(0)} \ast \tilde{\eta}(|\cdot|/h))(x).
\end{align*}
We choose the normalized kernel $K\coloneqq\tilde{K}/\|\tilde{K}\|_{L^1}$. Then $K\in C_c^{\infty}$ is point-symmetric and constant on $B_{1-2^{1/d}h}(0)$. Thus $(K \ast \chi_{E_0})(x)=1/2$ for $x\in B_{1-2^{-1/d}-2h}(0)$, which is a set of positive measure for $h$ small enough.
\end{example}
In these two examples we relied on the convolution kernel being constant on some part of their support. We conclude with an example, where the kernel is smooth and on its support strictly decreasing out of the origin, in the sense that for any $t<1$ and $x\in \supp(K)$ we have $K(tx)<K(x)$. 
\begin{example}
\label{ex: decreasing}
We construct the kernel such that
\begin{align*}
F(y) \coloneqq \int_{\{x_d = y\}}K(x)d\mathcal{H}^{d-1}(x)
\end{align*}
is constant on a centralized strip that contains more than half of the mass of the kernel. Then we can choose the set $E_0$ to be a disc of the proper thickness. For some large number $R$, let $E \coloneqq B^{d-1}_R(0) \times [-3/2,3/2]\|\tilde{\eta}\|_{L^1}$. We further set $\varphi(x) = 1 - \eta(x/(2\|\tilde{\eta}\|_{L^1}))/4$ and
\begin{align*}
\tilde{K}(X,y) \coloneqq \left\{ \begin{array}{lr} \varphi(y)^{1-d} \, \eta_{d-1}\left( \frac{X}{\varphi(y)} \right), &|y| \leq 2\|\tilde{\eta}\|_{L^1} \\
								\tilde{\eta}(y) \, \eta_{d-1}(X), &|y| > 2\|\tilde{\eta}\|_{L^1}. \end{array} \right.
\end{align*}
Then we choose the suitable kernel $K=\tilde{K}/\|\tilde{K}\|_{L^1}$ (see Figure~\ref{fig:decreasing kernel} for a plot in two dimensions). By shifting with $\varphi$, we notice that $F$ is constant on $[-2\|\tilde{\eta}\|_{L^1},2\|\tilde{\eta}\|_{L^1}]$, where more than half of the mass of $K$ is located. Thus we can shift the kernel in vertical direction without changing the convolution with the characteristic function of the disc. Indeed for the given set $E$ we find that
\begin{align*}
(K \ast \chi_{E})(x) = \frac{1}{2} \text{ on } {B_{R-1}^{d-1}(0) \times [-1/4,1/4]\|\tilde{\eta}\|_{L^1}},
\end{align*}
i.e., the convolution is $1/2$ on a thinner disc.
\end{example} 

\begin{figure}[!h]
\centering
\includegraphics[width=0.48\textwidth]{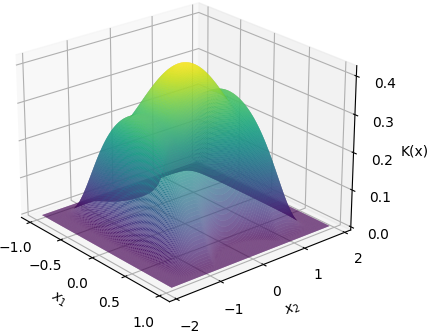}
\includegraphics[width=0.48\textwidth]{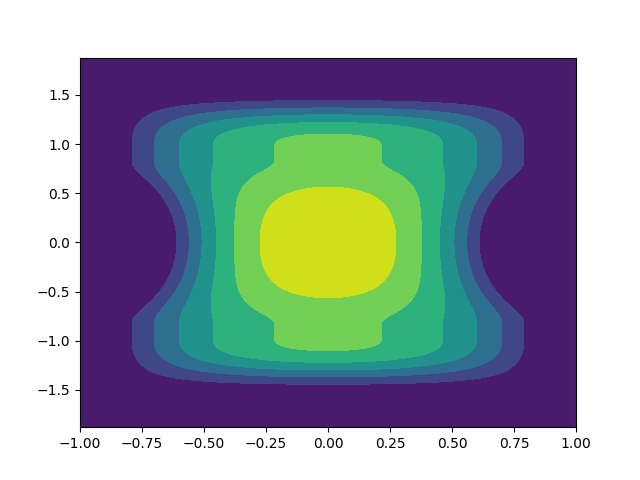}
\caption{Plot of $K$ in Example~\ref{ex: decreasing}}
\label{fig:decreasing kernel}
\end{figure}

\section*{Acknowledgments}
The present paper is an extension of the first author's master's thesis at the  University of Bonn.
This project has received funding from the Deutsche Forschungsgemeinschaft (DFG, German Research Foundation) under Germany's Excellence Strategy -- EXC-2047/1 -- 390685813.

\bibliographystyle{amsplain}
\bibliography{refs2}

\providecommand{\bysame}{\leavevmode\hbox to3em{\hrulefill}\thinspace}
\providecommand{\MR}{\relax\ifhmode\unskip\space\fi MR }
\providecommand{\MRhref}[2]{%
  \href{http://www.ams.org/mathscinet-getitem?mr=#1}{#2}
}
\providecommand{\href}[2]{#2}
\begin{thebibliography}{10}

\bibitem{Alb:Bell:opt:pro}
Giovanni Alberti and Giovanni Bellettini, \emph{A non-local anisotropic model
  for phase transitions}, Mathematische Annalen \textbf{310} (1998), 527--560.

\bibitem{Alb:Bell:asym:beh}
\bysame, \emph{A non-local anisotropic model for phase transitions: asymptotic
  behaviour of rescaled energies}, European Journal of Applied Mathematics
  \textbf{9} (1998), 261--284.

\bibitem{Alberti:Bianchini:Crippa}
Giovanni Alberti, Stefano Bianchini, and Gianluca Crippa, \emph{Structure of
  level sets and {S}ard-type properties of {L}ipschitz maps}, Annali della
  Scuola Normale Superiore di Pisa, Classe di Science \textbf{12} (2013),
  863--902.

\bibitem{Almgren:Taylor:Wang}
Fred Almgren, {Jean E.} Taylor, and Lihe Wang, \emph{Curvature-driven flows: A
  variational approach}, SIAM Journal on Control and Optimization \textbf{31}
  (1993), 387--438.

\bibitem{Barles:Georgelin}
Guy Barles and Christine Georgelin, \emph{A simple proof of convergence for an
  approximation scheme for computing motions by mean curvature}, SIAM Journal
  on Numerical Analysis \textbf{32} (1995), 484--500.

\bibitem{Cesaroni:Novaga}
Annalisa Cesaroni and Matteo Novaga, \emph{K mean-convex and {K}-outward
  minimizing sets}, preprint (2020), available at
  \texttt{https://arxiv.org/abs/2011.12614}.

\bibitem{Chambolle:Novaga}
Antonin Chambolle and Matteo Novaga, \emph{Anisotropic and crystalline mean
  curvature flow of mean-convex sets}, Annali della Scuola Normale Superiore di
  Pisa, Classe di Scienze, to appear, available at
  \texttt{https://arxiv.org/abs/2004.00270}.

\bibitem{Laux:Philippis}
Guido {De Philippis} and Tim Laux, \emph{Implicit time discretization for the
  mean curvature flow of outward minimizing sets}, Annali della Scuola Normale
  Superiore di Pisa, Classe di Scienze \textbf{21} (2020), 911--930.

\bibitem{Dro:Eym}
Jerome Droniou and Robert Eymard, \emph{Uniform-in-time convergence of
  numerical methods for non-linear degenerate parabolic equations}, Numerische
  Mathematik \textbf{132} (2016).

\bibitem{Elsey:Esedoglu}
Matt Elsey and Selim Esedo\u{g}lu, \emph{Threshold dynamics for anisotropic
  surface energies}, AMS Mathematics of Computations \textbf{87} (2018),
  1721--1756.

\bibitem{Esedoglu:Otto}
Selim Esedo\u{g}lu and Felix Otto, \emph{Threshold dynamics for networks with
  arbitrary surface tensions}, Communications on Pure and Applied Mathematics
  \textbf{68} (2015), no.~5, 808--864.

\bibitem{Evans}
Lawrence~C. Evans, \emph{Convergence of an algorithm for mean curvature
  motion}, Indiana University Mathematics Journal \textbf{42} (1993), 533--557.

\bibitem{Grayson}
{Matthew A.} Grayson, \emph{A short note on the evolution of a surface by its
  mean curvature}, Duke Mathematical Journal \textbf{58} (1989), 555--558.

\bibitem{Huisken}
Gerhard Huisken, \emph{Flow by mean curvature of convex surfaces into spheres},
  Journal of Differential Geometry \textbf{20} (1984), 237--266.

\bibitem{Ishii:Pires:Souganidis}
Hitoshi Ishii, Gabriel~E. Pires, and Panagiotis~E. Souganidis, \emph{Threshold
  dynamics type approximation schemes for propagating fronts}, Journal of the
  Mathematical Society of Japan \textbf{51} (1999), no.~2, 267--308.
  \MR{1674750}

\bibitem{Laux:Lelmi}
Tim Laux and Jona Lelmi, \emph{De {G}iorgi's inequality for the thresholding
  scheme with arbitrary mobility and surface tensions}, to appear in Calculus
  of Variations and Partial Differential Equations, preprint available at
  \texttt{https://arxiv.org/abs/2101.11663}.

\bibitem{Laux:Otto}
Tim Laux and Felix Otto, \emph{Convergence of the thresholding scheme for
  multi-phase mean-curvature flow}, Calculus of Variations and Partial
  Differential Equations \textbf{55} (2016).

\bibitem{Laux:Otto:Brakke}
\bysame, \emph{Brakke's inequality for the thresholding scheme}, Calculus of
  Variations and Partial Differential Equations \textbf{59} (2020).

\bibitem{Laux:Otto:de:giorgi}
\bysame, \emph{The thresholding scheme for mean curvature flow and de
  {G}iorgi's ideas for minimizing movements}, The Role of Metrics in the Theory
  of Partial Differential Equations, Mathematical Society of Japan, 2020,
  pp.~63--93.

\bibitem{Luckhaus:Sturzenhecker}
Stephan Luckhaus and Thomas Sturzenhecker, \emph{Implicit time discretization
  for the mean curvature flow equation}, Calculus of Variations and Partial
  Differential Equations \textbf{3} (1995), 253--271.

\bibitem{Maggi}
Francesco Maggi, \emph{Sets of finite perimeter and geometric variational
  problems}, Cambridge University Press, 2012.

\bibitem{Mascarenhas}
Pierre Mascarenhas, \emph{Diffusion generated motion by mean curvature}, CAM
  Report 92-33 (1992).

\bibitem{MBO}
Barry Merriman, James Bence, and Stanley Osher, \emph{Diffusion generated
  motion by mean curvature}, CAM Report 92-18 (1992).

\end{thebibliography}
\end{document}